\documentclass{article}
\usepackage[T1]{fontenc}
\usepackage{geometry}
\usepackage{mathrsfs}
\usepackage{amsthm}
\usepackage{amsmath}
\usepackage{amssymb}
\usepackage{esint}
\usepackage[unicode=true,pdfusetitle,
 bookmarks=true,bookmarksnumbered=false,bookmarksopen=false,
 breaklinks=false,pdfborder={0 0 1},backref=false,colorlinks=false]
 {hyperref}

\makeatletter
\numberwithin{equation}{section}
\newcommand{\lyxaddress}[1]{
\par {\raggedright #1
\vspace{1.4em}
\noindent\par}
}
\theoremstyle{plain}
\newtheorem{thm}{\protect\theoremname}[section]
  \theoremstyle{remark}
  \newtheorem*{notation*}{\protect\notationname}
  \theoremstyle{plain}
  \newtheorem{lem}[thm]{\protect\lemmaname}
  \theoremstyle{definition}
  \newtheorem{defn}[thm]{\protect\definitionname}
  \theoremstyle{plain}
  \newtheorem{prop}[thm]{\protect\propositionname}
  \theoremstyle{remark}
  \newtheorem{rem}[thm]{\protect\remarkname}
  \theoremstyle{remark}
  \newtheorem*{acknowledgement*}{\protect\acknowledgementname}

\@ifundefined{date}{}{\date{}}
\usepackage{enumitem}

\usepackage{pgf}
\usepackage{tikz}
\usepackage{verbatim}

\usetikzlibrary{arrows,automata}
\usetikzlibrary{positioning}

\tikzset{
    state/.style={
           rectangle,
           rounded corners,
           draw=black, very thick,
           minimum height=2em,
           inner sep=2pt,
           text centered,
           },
}

\newcommand{\dce}{\mathfrak{E}}
\newcommand{\dct}{\mathfrak{T}}

\makeatother

  \providecommand{\acknowledgementname}{Acknowledgement}
  \providecommand{\definitionname}{Definition}
  \providecommand{\lemmaname}{Lemma}
  \providecommand{\notationname}{Notation}
  \providecommand{\propositionname}{Proposition}
  \providecommand{\remarkname}{Remark}
\providecommand{\theoremname}{Theorem}

\begin{document}
\title{A stable self-similar singularity of evaporating drops: ellipsoidal collapse to a point}
\author{
Marco A. Fontelos\\
Seok Hyun Hong\\
Hyung Ju Hwang
}
\maketitle

\lyxaddress{Instituto de Ciencias Matemáticas (ICMAT), C/Nicolás Cabrera, 28049
Madrid, Spain\\ E-mail: marco.fontelos@icmat.es}

\lyxaddress{Department of Mathematics, Pohang University of Science and Technology,
Pohang 790-784, Republic of Korea\\ E-mail: inrhg@postech.ac.kr}

\lyxaddress{Department of Mathematics, Pohang University of Science and Technology,
Pohang 790-784, Republic of Korea\\ E-mail: hjhwang@postech.ac.kr}
\begin{abstract}
We study the problem of evaporating drops contracting to a point.
Going back to Maxwell and Langmuir, the existence of a spherical solution
for which evaporating drops collapse to a point in a self-similar
manner is well established in the physical literature. The diameter
of the drop follows the so-called $D^{2}$ law: the second power of
the drop-diameter decays linearly in time. In this study we provide
a complete mathematical proof of this classical law. We prove that
evaporating drops which are initially small perturbations of a sphere
collapse to a point and the shape of the drop converges to a self-similar
ellipsoid whose center, orientation, and semi-axes are determined
by the initial shape.
\end{abstract}
\tableofcontents{}

\section{Introduction}

This paper addresses the question on the leading-order mechanism of
the dynamics of evaporating drops. This question has attracted a great
deal of attention in the physics society. The two situations that
have been studied most are the ``coffee ring'' problem \cite{Deegan:1997vb}
and the evaporation of a completely wetting liquid on a perfectly
flat surface (\cite{Bonn:2009ha} and references therein). Our study
is motivated by the case where a drop wets completely but evaporation
is significant. Many experiments reveal that the radius of such a
drop goes to zero in finite time, with a characteristic scaling exponent
close to $1/2$ (see \cite{Cachile:2002dl}, \cite{Cachile:2002ft},
\cite{Poulard:2003kx}). The question is thus on the origin of the
exponent $1/2$.

Maxwell \cite{Maxwell:1877tf} and Langmuir \cite{PhysRev.12.368}
studied evaporation based on the following simplifying hypotheses:
the evaporation process is diffusion-controlled, quasi-stationary,
and isothermal; the interface of the drop is at local equilibrium
(see \cite{Erbil:2012gy} for an historical account). These assumptions
and the spherical symmetry lead to the classical $D^{2}$ law \cite{Fuchs:1959vw}:
the square of the drop-diameter $D$ decreases linearly in time. In
fact, such hypotheses are essentially related to the $1/2$ exponent
because the $D^{2}$ law is exactly the case of the scaling exponent
$1/2$. Moreover, in many situations the evaporation of drops is well
described by these hypotheses (e.g. \cite{Deegan:1997vb}, \cite{Hu:2002hs},
\cite{Eggers:2010fh}, \cite{Gelderblom:2012by}). The assumptions
can therefore be regarded as a leading-order approximation of more
complex behavior (see \cite{ShahidzadehBonn:2006fb}, \cite{Bonn:2009ha},
\cite{Eggers:2010fh}, \cite{Cazabat:2010bd}, and references therein).

In this study we provide a complete mathematical proof of the $D^{2}$
law built on the simplifying hypotheses. The spherical symmetry will
not be assumed and it will be verified that evaporating drops initially
close to a sphere collapse to a point in finite time. The way to collapse
obeys the $D^{2}$ law.

We first introduce a precise mathematical formulation of the problem.
Consider a drop that occupies an open, bounded, and simply connected
domain $\Omega(t)\subset\mathbb{R}^{3}$, where $t$ denotes the time-variable.
In the outside of the drop, there is the vapor concentration $\phi(\mathbf{x},t)$,
$\mathbf{x}\in\mathbb{R}^{3}$. The characteristic time of the evaporation
process is much larger than the time that characterizes transfer rates
across the interface of the drop. As a result, the concentration satisfies
Laplace's equation at each time, i.e. $\Delta\phi=0$. Moreover, the
concentration reaches local equilibrium at the drop interface, i.e.
$\phi=1$ on the boundary of $\Omega(t)$ without loss of generality.
Then the concentration far from the drop should be lower than 1, and
hence we assume that $\phi\rightarrow0$ as $|\mathbf{x}|\rightarrow\infty$.
In the inside of the drop, the liquid satisfies the Stokes equations
because the Reynolds number of the drop is sufficiently small. Since
the process is isothermal, only the normal stress must be balanced
at the interface. The two conservation laws hold for the drop: linear
momentum and angular momentum. The liquid-vapor interface will be
moved by the conservation of mass at the interface. Here, the evaporative
flux is proportional to the concentration gradient, i.e. $\frac{1}{2}\frac{\partial\phi}{\partial\mathbf{n}}$,
where $\frac{1}{2}$ is just for the simplicity of calculations and
$\mathbf{n}$ is the outward normal of the boundary of $\Omega(t)$.
Finally we are led to the following system:
\begin{equation}
\begin{cases}
-\nabla p+\Delta\mathbf{u}=0 & \textnormal{in}\,\Omega(t),\\
\operatorname{div}\mathbf{u}=0 & \textnormal{in}\,\Omega(t),\\
\mathbf{T}[\mathbf{u},p]\mathbf{n=-\sigma\kappa\mathbf{n}} & \textnormal{on}\,\Gamma(t),\\
\int_{\Omega(t)}\mathbf{u}dV=0,\\
\int_{\Omega(t)}(\mathbf{u}\times\mathbf{x})dV=0,\\
\Delta\phi=0 & \textnormal{in}\,\mathbb{R}^{3}\setminus\overline{\Omega}(t),\\
\phi=1 & \textnormal{on}\,\Gamma(t),\\
\phi\rightarrow0 & \textnormal{as}\,|\mathbf{x}|\rightarrow\infty,
\end{cases}\label{eq:ED system primary}
\end{equation}
\begin{align}
 & v_{\mathbf{n}}=\mathbf{u}\cdot\mathbf{n}+\frac{1}{2}\frac{\partial\phi}{\partial\mathbf{n}}\quad\textnormal{on}\,\Gamma(t),\label{eq:ED kinematic condition primary}\\
 & \Gamma(0)=\Gamma^{0},\label{eq:ED initial condition primary}
\end{align}
where $\Gamma(t)$ denotes the boundary of $\Omega(t)$, $\Gamma^{0}$
is the initial shape of the free boundary, $\mathbf{T}[\mathbf{u},p]=-p\mathbf{I}+\nabla\mathbf{u}+(\nabla\mathbf{u})^{T}$
is the stress tensor, $\sigma$ is a given positive constant, $\kappa$
is the mean curvature of $\Gamma(t)$ ($\kappa$ is positive for a
sphere), and $v_{\mathbf{n}}$ is the velocity of the free boundary
in the direction of the outward normal.

From the mathematical point of view, the problem (\ref{eq:ED system primary})-(\ref{eq:ED initial condition primary})
belongs to a class of free-boundary problems. Similar mathematical
problems arise in various contexts in physics and biology: displacement
of oil by water in porous media \cite{ISI:000187862600010}, evolution
of uncharged or charged droplets (\cite{Gunther:1997gs}, \cite{Friedman:2002io},
\cite{Friedman:2002en}, \cite{Fontelos:2004cz}), evolution of tumors
(\cite{Bazaliy:2003ic}, \cite{Fontelos:2003wm}, \cite{Friedman:2006bt},
\cite{Friedman:2006ce}, \cite{Friedman:2007bw}, \cite{Friedman:2008ei}),
melting of spherical crystals \cite{Herrero:1997te}, loop dislocations
in crystals \cite{Friedman:1992jw}, Stefan problems (\cite{Chen:1993dv},
\cite{Constantin:1993vw}, \cite{Escher:1998gj}, \cite{Friedman:2001vk},
\cite{Pruss:2008bt}, \cite{Hadzic:2011hn}, \cite{Pruss:2013fu}),
and Ostwald ripening \cite{Alikakos:2003fo}, just to name a few.
However, there is another interesting mathematical aspect in evaporating
drops: formation of a finite-time singularity. Evaporating drops lose
their mass but they retain constant density (we have assumed that
the process is isothermal). As a consequence, the drop vanishes in
a finite time and the boundary of the drop might collapse to a point,
a collection of points, or a more complicated geometric object. Since
there is a loss of regularity of the boundary in this process, a singularity
happens and one has to resolve it with available mathematical tools
\cite{Eggers:2009ha}.

The system (\ref{eq:ED system primary})-(\ref{eq:ED initial condition primary})
has a self-similar solution given by 
\begin{equation}
\begin{cases}
p=\sigma(1-t)^{-1/2},\\
\mathbf{u}=\mathbf{0},\\
\phi=|\mathbf{x}|^{-1}(1-t)^{1/2},\\
\Gamma(t)=\{|\mathbf{x}|=(1-t)^{1/2}\}.
\end{cases}\label{eq:ED self-similar solution}
\end{equation}
The solution (\ref{eq:ED self-similar solution}) implies that the
drop vanishes at the extinction time $t=1$. One can easily check
that the only self-similar solutions of (\ref{eq:ED system primary})-(\ref{eq:ED initial condition primary})
with a power-law structure are the ones given by (\ref{eq:ED self-similar solution}).
Moreover, $|\mathbf{x}|^{2}=(1-t)$ means that $D^{2}$ decays linearly
in time, i.e. the $D^{2}$ law is exactly satisfied by (\ref{eq:ED self-similar solution}).
The main aim of this study is therefore to prove the stability of
the self-similar solution (\ref{eq:ED self-similar solution}).

We now state our main result. We will begin with some notations. In
this paper we will formulate the interface of the drop in terms of
spherical coordinate systems, for example,
\[
\Gamma(t)=\left\{ r=f(\theta,\varphi,t)\right\} ,
\]
where $(r,\theta,\varphi)$ is a spherical coordinate system with
respect to the origin and $f$ is a real function whose variables
$(\theta,\varphi)$ are defined on $\mathbb{S}=[0,\pi]\times[0,2\pi)$.
Then $H^{s}(\mathbb{S})$ with $s\geq0$ denote the Sobolev spaces.
The spherical harmonics $Y_{l,m}$ are on $\mathbb{S}$, and $f=\sum_{l,m}f_{l,m}Y_{l,m}$
is the spherical harmonic expansion as defined in Appendix \ref{sec:S Harmonics and Vector S Harmonics}.
Finally the main result reads as follows: the self-similar solution
(\ref{eq:ED self-similar solution}) is stable under small perturbations
in $H^{s}(\mathbb{S})$, for some appropriate $s\geq0$. Moreover,
the drop shape near extinction is universally an ellipsoid.
\begin{thm}[Stable ellipsoidal collapse to a point]
\label{thm:stable ellipsoidal collapse to a point}Assume that the
initial shape $\Gamma^{0}$ is given by $\{r=1+\epsilon g^{0}(\theta,\varphi)\}$,
where $g^{0}\in H^{6}(\mathbb{S})$. Then, for a sufficiently small
$\epsilon>0$, there exists a unique solution to the free-boundary
problem (\ref{eq:ED system primary})-(\ref{eq:ED initial condition primary})
such that 
\begin{equation}
\frac{\Gamma(t)-\epsilon\mathbf{x}_{0}}{R(t)}=\left\{ r=1+\epsilon g^{\xi}(\theta,\varphi,t)\right\} ,\quad R(t)=\left[1-\frac{t}{1+\epsilon t_{0}(2\sqrt{\pi})^{-1}}\right]^{1/2}\label{eq:thm1.1 free boundary solution}
\end{equation}
on the time interval $0\leq t<1+\epsilon t_{0}(2\sqrt{\pi})^{-1}$
for some $(\mathbf{x}_{0},t_{0})=(\mathbf{x}_{0}(\epsilon),t_{0}(\epsilon))\in\mathbb{R}^{3}\times\mathbb{R}$,
and we have:

\begin{enumerate}[label=(\alph*)]

\item The solution $\left\{ r=1+\epsilon g^{\xi}(\theta,\varphi,t)\right\} $,
as well as the corresponding $(\mathbf{u},p,\phi)$, of the transformed
problem is unique in the function space $\mathscr{X}$ to be defined
in Theorem \ref{thm:the nonlinear evolution problem}.

\item The quantity $(\mathbf{x}_{0},t_{0})$ satisfies
\begin{equation}
|\mathbf{x}_{0}-\tilde{\mathbf{x}}_{0}|=O(\epsilon),\quad|t_{0}-\tilde{t}_{0}|=O(\epsilon),\label{eq:x0 t0 thm1.1}
\end{equation}
where $(\tilde{\mathbf{x}}_{0},\tilde{t}_{0})\in\mathbb{R}^{3}\times\mathbb{R}$
merely depends on the $l\leq2$ terms in the spherical harmonic expansion
of $g^{0}$.

\item There exist constants $C>0$ and $0<\lambda_{0}<1$ such that
\begin{equation}
\sup_{(\theta,\varphi)\in\mathbb{S}}|g^{\xi}(\theta,\varphi,t)-g^{\mathfrak{E}}(\theta,\varphi)|\leq CR^{\lambda_{0}}(t),\quad\textnormal{for all}\,\, t\in\left[0,1+\epsilon t_{0}(2\sqrt{\pi})^{-1}\right),\label{eq:thm1.1 converging to ellipsoid}
\end{equation}
where $\left\{ r=1+\epsilon g^{\mathfrak{E}}(\theta,\varphi)\right\} $
is an ellipsoid which satisfies
\begin{equation}
\sup_{(\theta,\varphi)\in\mathbb{S}}|g^{\mathfrak{E}}(\theta,\varphi)-g_{0,0}^{0}Y_{0,0}-\sum_{|m|\leq2}e^{-\sigma\frac{20}{19}}g_{2,m}^{0}Y_{2,m}|=O(\epsilon).\label{eq:gdce thm1.1}
\end{equation}
\end{enumerate}
\end{thm}
The proof of Theorem \ref{thm:stable ellipsoidal collapse to a point}
will be given at the end of Section \ref{sec:The linear evolution problem}.
In Section \ref{sec:Main theorems for the reformulated problem} we
will state the corresponding result to Theorem \ref{thm:stable ellipsoidal collapse to a point}
for the reformulated problem on the fixed domain. In Section \ref{sec:Reformulation of the problem}
we will introduce this reformulation.

Theorem \ref{thm:stable ellipsoidal collapse to a point} confirms
that evaporating drops with arbitrary shapes close to a sphere also
obey the $D^{2}$ law and the asymptotic shape of the drop is generically
an ellipsoid. To our knowledge, this is the first result reporting
the robustness of the $D^{2}$ law and its ellipsoidal structure near
extinction.

Evaporation of macroscopic drops is well described by the simplifying
hypotheses for the main part of their lifetimes, as discussed in \cite{Cazabat:2010bd}.
At the very end of their lifetimes, we should consider another physics,
e.g. changes of the equilibrium concentration at the interface, thermal
effects, etc. However, self-similar singularities are intermediate-asymptotic
solutions \cite{Barenblatt:1996tm}. Hence, it could conceivably be
hypothesized that, when macroscopic drops evaporate, the ellipsoidal
structure near extinction is a building block in the transitional
period between the macroscopic scale and the much smaller scale. The
$D^{2}$ law for micrometer-sized or nanometer-sized droplets is discussed
in \cite{Holyst:2013gq} and references therein.

Theorem \ref{thm:stable ellipsoidal collapse to a point} has only
examined evaporating drops which are freely suspended. However, in
many situations where an evaporating drop is on a solid substrate,
a characteristic scaling exponent of the radius of the drop is close
to $1/2$ (e.g. \cite{Cachile:2002dl}, \cite{Cachile:2002ft}, \cite{Poulard:2003kx},
\cite{Poulard:2005bd}, \cite{ShahidzadehBonn:2006fb}, \cite{Bonn:2009ha},
\cite{Eggers:2010fh}). It may be the case therefore that the simplifying
hypotheses are major factors but intricate conditions effecting such
deviation happen, for example, at the contact line. The result of
this study, as a basis for more involved analyses, may help us to
understand such a phenomenon. Future work is required to establish
this.

\subsection{Overview of the proof}

The key part of proving Theorem \ref{thm:stable ellipsoidal collapse to a point}
is to treat the ellipsoidal solution $g^{\dce}$ as an operator of
certain parameters (i.e. $g^{\dce}$ is determined completely by the
parameters), and to construct a stable self-similar singularity of
(\ref{eq:ED system primary})-(\ref{eq:ED initial condition primary}),
which approaches the ellipsoidal solution in the rescaled domain.

In Section \ref{sec:Reformulation of the problem} we introduce self-similar
variables (\cite{Giga:1985bm}, \cite{Giga:1987en}) and the Hanzawa
transformation \cite{Hanzawa:1981ka}, and hence the system (\ref{eq:ED system primary})-(\ref{eq:ED initial condition primary})
can be transformed into a fixed domain. Moreover, we decompose the
transformed system into the ellipsoidal part $g^{\dce}$ and the evolution
part $g^{\dct}$. By this decomposition, the rescaled solution $g^{\xi}$
of (\ref{eq:ED system primary})-(\ref{eq:ED initial condition primary})
can be regarded as a perturbation of the ellipsoidal solution.

In Section \ref{sec:Main theorems for the reformulated problem} we
state the main theorems in the reformulated framework. Theorem \ref{thm:existence of ellipsoids}
provides the existence and uniqueness of the ellipsoidal solution
that depends on certain parameters. In fact, the starting point for
the decomposition defined in Section \ref{sec:Reformulation of the problem}
is at Theorem \ref{thm:existence of ellipsoids}. Theorem \ref{thm:the linear evolution problem}
analyzes the auxiliary linear problem for the evolution part of the
decomposition. By Theorem \ref{thm:the linear evolution problem},
we can deduce the parameters that determine the ellipsoidal solution.
Moreover, Theorem \ref{thm:the linear evolution problem} shows that
the ellipsoidal solution is approached exponentially in the logarithmic
variable $\tau=-\ln R(t)$. Finally, we prove Theorem \ref{thm:the nonlinear evolution problem}
for the nonlinear evolution problem of the decomposition. By reversing
the procedure in Section \ref{sec:Reformulation of the problem},
it readily follows that Theorem \ref{thm:the nonlinear evolution problem}
is equivalent to Theorem \ref{thm:stable ellipsoidal collapse to a point}.

In Section \ref{sec:the ODE system} we solve the system of ordinary
differential equations deduced from the reformulated (\ref{eq:ED system primary}),
in terms of the spherical harmonics and the vector spherical harmonics.
By using the results in this section, we can compute the eigenvalues
of certain linearized systems on $\mathbb{S}$ for $g^{\dce}$ and
$g^{\dct}$, respectively.

In Section \ref{sec:Existence of ellipsoids} we prove Theorem \ref{thm:existence of ellipsoids}.
By using the rescaling procedure in Section \ref{sec:Reformulation of the problem},
one can easily check that the evaporative flux $\frac{1}{2}\frac{\partial\phi}{\partial\mathbf{n}}$
overwhelms $\mathbf{u}\cdot\mathbf{n}$ in (\ref{eq:ED kinematic condition primary})
near extinction. Consequently, the eigenvalues deduced from $\frac{1}{2}\frac{\partial\phi}{\partial\mathbf{n}}$
are dominant as $\tau\rightarrow\infty$. However, such eigenvalues
contain a vanishing eigenvalue corresponding to $Y_{2,m}$. This means
that we have to look at the nonlinear behavior to establish the dynamics
in the $Y_{2,m}$-direction, which is far from being trivial. We remark
that the existence of this vanishing eigenvalue is not connected to
a symmetry-breaking bifurcation like those found in free-boundary
problems studied recently in connection to drops, tumor growth, etc.
(cf. \cite{Fontelos:2003wm}, \cite{Friedman:2001ct}, \cite{Friedman:2001fa}
as well as the papers \cite{Fontelos:2004cz}, \cite{Borisovich:2005cp},
\cite{Friedman:2007bw}, \cite{Friedman:2008ei} making use of the
Crandall-Rabinowitz theorem \cite{Crandall:1971wd}). Hence, for this
problem we have to rely on new ideas and techniques. Specifically
we use a property of null quadrature domains, i.e. the boundary of
$\Omega$ is an ellipsoid if and only if the gravity induced by the
uniform mass on the complement of $\Omega$ is equal to zero in $\Omega$
(\cite{Kellogg:1967uz}, \cite{DiBenedetto:1986jt}, \cite{Friedman:1986hn}).

In Section \ref{sec:The linear evolution problem} we prove Theorem
\ref{thm:the linear evolution problem}. By the results in Section
\ref{sec:the ODE system}, we have the system of ordinary differential
equations for $g^{\dct}$. The stability is controlled by the eigenvalues
of the linear operator of the system. However, there are two positive
eigenvalues corresponding to $Y_{0,0}$ and $Y_{1,m}$, respectively.
These positive eigenvalues are related to a small shifting on the
position of the singularity in space and time, but do not indicate
instability, as noted in \cite{Filippas:1992eq}, \cite{Velazquez:1991ue}.
Therefore, we determine the singularity position such that the artifact-instability
does not happen. Moreover, the eigenvalue corresponding to $Y_{2,m}$
vanishes. The motivation of the decomposition in Section \ref{sec:Reformulation of the problem}
and Theorem \ref{thm:existence of ellipsoids} is fundamentally based
on this vanishing eigenvalue. Here, the ellipsoidal solution is determined
by the parameters in $Y_{2,m}$-direction such that $g^{\dct}$ converges
to zero exponentially as $\tau\rightarrow\infty$.

In Appendix \ref{sec:S Harmonics and Vector S Harmonics} we introduce
properties of the spherical harmonics and the vector spherical harmonics,
which are necessary for our proofs.
\begin{notation*}
We will use the following notations throughout the paper: $H^{s}(\Omega)$
denotes the Sobolev space of functions in $\Omega$; the weighted
Sobolev space $H^{s}(\Omega;w)$ equipped with the norm 
\[
\|G\|_{H^{s}(\Omega;w)}^{2}=\sum_{|\alpha|\leq s}\int_{\Omega}|D^{\alpha}G|^{2}wdV,
\]
where 
\[
w(\mathbf{x})=\begin{cases}
1/|\mathbf{x}|^{2} & \textnormal{if}\quad\alpha=0,\\
1 & \textnormal{if}\quad\alpha>0
\end{cases}
\]
holds; $\overline{a}$ is the complex conjugate of $a\in\mathbb{C}$.
\end{notation*}

\section{\label{sec:Reformulation of the problem}Reformulation of the problem}

In this section we transform the free-boundary problem into the fixed
domain by using self-similar variables (\cite{Giga:1985bm}, \cite{Giga:1987en})
and the Hanzawa transformation \cite{Hanzawa:1981ka}. Moreover, we
introduce a decomposition of the solution into the two components
that will be studied separately: the ellipsoidal part and the evolution
part.

As shown in \cite{Friedman:2002io}, it will be convenient to replace
(\ref{eq:ED system primary})-(\ref{eq:ED initial condition primary})
by the following system:

\begin{equation}
\begin{cases}
\Delta p=0 & \textnormal{in}\,\Omega(t),\\
-\nabla p+\Delta\mathbf{u}=0 & \textnormal{in}\,\Omega(t),\\
\operatorname{div}\mathbf{u}=0 & \textnormal{on}\,\Gamma(t),\\
\mathbf{T}[\mathbf{u},p]\mathbf{n=-\sigma\kappa\mathbf{n}} & \textnormal{on}\,\Gamma(t),\\
\int_{\Omega(t)}\mathbf{u}dV=0,\\
\int_{\Omega(t)}(\mathbf{u}\times\mathbf{x})dV=0,\\
\Delta\phi=0 & \textnormal{in}\,\mathbb{R}^{3}\setminus\overline{\Omega}(t),\\
\phi=1 & \textnormal{on}\,\Gamma(t),\\
\phi\rightarrow0 & \textnormal{as}\,|\mathbf{x}|\rightarrow\infty,
\end{cases}\label{eq:ED system primary-1}
\end{equation}
\begin{align}
 & v_{\mathbf{n}}=\mathbf{u}\cdot\mathbf{n}+\frac{1}{2}\frac{\partial\phi}{\partial\mathbf{n}}\quad\textnormal{on}\,\Gamma(t),\label{eq:ED kinematic condition primary-1}\\
 & \Gamma(0)=\Gamma^{0}.\label{eq:ED initial condition primary-1}
\end{align}

\subsection{Reduction to the fixed domain}

In this subsection we will use three different notations of variables
in terms of spherical coordinate systems. The initial shape $\Gamma^{0}$
is given by $\{r^{\prime}=1+\epsilon g^{0}(\theta^{\prime},\varphi^{\prime})\}$,
here $(r',\theta',\varphi')$ is the spherical coordinate system for
the shrinking domain. Applying self-similar variables to the shrinking
domain, we deduce the spherical coordinate system $(r^{\xi},\theta^{\xi},\varphi^{\xi})$
for the rescaled domain which still has the free boundary. Finally,
by using the Hanzawa transformation, we arrive at the spherical coordinate
system $(r,\theta,\varphi)$ for the fixed domain. We give a simple
diagram of these changes of variables:

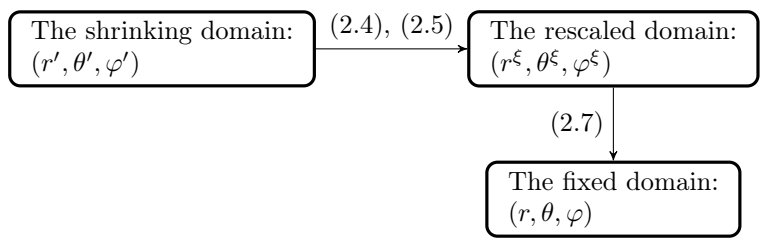
\begin{center}
\begin{tikzpicture}[->,>=stealth']

\node[state](Shrinking)
{
\begin{tabular}{l}
The shrinking domain:\\
$(r',\theta',\varphi')$
\end{tabular}
};

\node[state, node distance=6cm, right of=Shrinking](Rescaled)
{
\begin{tabular}{l}
The rescaled domain:\\
$(r^\xi, \theta^\xi, \varphi^\xi)$
\end{tabular}
};

\node[state,  below of=Rescaled,  node distance=2cm](Fixed)
{
\begin{tabular}{l}
The fixed domain:\\
$(r,\theta,\varphi)$
\end{tabular}
};

\path 

(Shrinking) edge node[above]
{
(\ref{eq:self-similar variables solution}), (\ref{eq:self-similar variables})
} (Rescaled)  

(Rescaled) edge               node[left]
{
(\ref{eq:hanzawa transformation})
} (Fixed)

; 

\end{tikzpicture}
\end{center}

As a first step, we introduce self-similar variables which reflect
the scaling-property of (\ref{eq:ED system primary-1})-(\ref{eq:ED initial condition primary-1})
as well as the fact that the position of the singularity in space
and time will be perturbed when the perturbed sphere collapses:
\begin{equation}
\begin{cases}
p(\mathbf{x},t)=\left(\sigma+\epsilon P^{\xi}(\xi,\tau)\right)/R(t),\\
\mathbf{u}(\mathbf{x},t)=\epsilon\mathbf{U}^{\xi}(\xi,\tau),\\
\phi(\mathbf{x},t)=R(t)/|\mathbf{x}-\epsilon\mathbf{x}_{0}|+\epsilon\Phi^{\xi}(\xi,\tau),\\
\left(\Gamma(t)-\epsilon\mathbf{x}_{0}\right)/R(t)=\left\{ r^{\xi}=1+\epsilon g^{\xi}(\theta^{\xi},\varphi^{\xi},\tau)\right\} ,
\end{cases}\label{eq:self-similar variables solution}
\end{equation}
with 
\begin{equation}
\tau=-\ln R(t)\quad\textnormal{and}\quad\xi=(\mathbf{x}-\epsilon\mathbf{x}_{0})/R(t),\label{eq:self-similar variables}
\end{equation}
where 
\[
R(t)=[1-t/(1+\epsilon t_{0}Y_{0,0})]^{1/2},
\]
the position of the singularity in space and time is given by $(\epsilon\mathbf{x}_{0},1+\epsilon t_{0}Y_{0,0})$,
both $\mathbf{x}_{0}$ and $t_{0}$ will be determined as a part of
the solution, and $(r^{\xi},\theta^{\xi},\varphi^{\xi})$ is the spherical
coordinate system for the three-dimensional $\xi$-space.

Then $(P^{\xi},\mathbf{U}^{\xi},\Phi^{\xi},g^{\xi})$ satisfies 
\begin{equation}
\begin{cases}
\Delta P^{\xi}=0 & \textnormal{in}\,\Omega^{\xi}(\tau),\\
-\nabla P^{\xi}+\Delta\mathbf{U}^{\xi}=0 & \textnormal{in}\,\Omega^{\xi}(\tau),\\
\operatorname{div}\mathbf{\mathbf{U}}^{\xi}=0 & \textnormal{on}\,\Gamma^{\xi}(\tau),\\
\mathbf{T}[\mathbf{U}^{\xi},P^{\xi}]\mathbf{n}^{\xi}=-\sigma(\kappa^{\xi}-1)\epsilon^{-1}\mathbf{n}^{\xi} & \textnormal{on}\,\Gamma^{\xi}(\tau),\\
\int_{\Omega^{\xi}(\tau)}\mathbf{U}^{\xi}dV=0,\\
\int_{\Omega^{\xi}(\tau)}(\mathbf{U}^{\xi}\times\xi)dV=0,\\
\Delta\Phi^{\xi}=0 & \textnormal{in}\,\mathbb{R}^{3}\setminus\overline{\Omega^{\xi}}(\tau),\\
\Phi^{\xi}=g^{\xi}/(1+\epsilon g^{\xi}) & \textnormal{on}\,\Gamma^{\xi}(\tau),\\
\Phi^{\xi}\rightarrow0 & \textnormal{as}\,|\xi|\rightarrow\infty,
\end{cases}\label{eq:ED rescaled system xi}
\end{equation}
where $\Gamma^{\xi}(\tau)=\left\{ r^{\xi}=1+\epsilon g^{\xi}(\theta^{\xi},\varphi^{\xi},\tau)\right\} $
holds, $\kappa^{\xi}$ is the mean curvature of $\Gamma^{\xi}(\tau)$,
and $\mathbf{n}^{\xi}$ is the outward normal of the free boundary
$\Gamma^{\xi}(\tau)$. The kinematic condition (\ref{eq:ED kinematic condition primary-1})
and the initial condition (\ref{eq:ED initial condition primary-1})
will be treated at the end of this subsection.

We now transform the free-boundary system (\ref{eq:ED rescaled system xi})
into the fixed domain by virtue of the Hanzawa transformation. The
transformation is 
\begin{equation}
r^{\xi}=r+\Psi(1-r)\epsilon g^{\xi}(\theta^{\xi},\varphi^{\xi},\tau),\quad(\theta^{\xi},\varphi^{\xi})=(\theta,\varphi),\label{eq:hanzawa transformation}
\end{equation}
with a cut-off function $\Psi(z)\in C_{0}^{\infty}$, 
\[
\Psi(z)=\begin{cases}
0,\quad\textnormal{if}\,|z|\geq3\delta_{0}/4\\
 & ,\quad\quad\left\vert \frac{d^{k}\Psi}{dz^{k}}\right\vert \leq\frac{C}{\delta_{0}^{k}},\\
1,\quad\textnormal{if}\,|z|<\delta_{0}/4
\end{cases}
\]
where $\delta_{0}$ is positive and small. It maps the domain bounded
by $\Gamma^{\xi}(\tau)$ into the sphere $\mathbb{B}:=\{r<1\}$. In
terms of the new spherical coordinate system $(r,\theta,\varphi)$,
by setting $u^{\xi}(r^{\xi},\theta^{\xi},\varphi^{\xi},\tau)=u(r,\theta,\varphi,\tau)$
we are led to 
\[
\frac{\partial u^{\xi}}{\partial r^{\xi}}=\frac{\partial r}{\partial r^{\xi}}\frac{\partial u}{\partial r},\quad\quad\frac{\partial^{2}u^{\xi}}{\partial(r^{\xi})^{2}}=\left(\frac{\partial r}{\partial r^{\xi}}\right)^{2}\frac{\partial^{2}u}{\partial r^{2}}+\frac{\partial^{2}r}{\partial(r^{\xi})^{2}}\frac{\partial u}{\partial r},
\]
\[
\frac{\partial u^{\xi}}{\partial\theta^{\xi}}=\frac{\partial u}{\partial\theta}+\frac{\partial r}{\partial\theta^{\xi}}\frac{\partial u}{\partial r},\quad\quad\frac{\partial^{2}u^{\xi}}{\partial(\theta^{\xi})^{2}}=\frac{\partial^{2}u}{\partial\theta^{2}}+2\frac{\partial r}{\partial\theta^{\xi}}\frac{\partial^{2}u}{\partial r\partial\theta}+\frac{\partial^{2}r}{\partial(\theta^{\xi})^{2}}\frac{\partial u}{\partial r}+\left(\frac{\partial r}{\partial\theta^{\xi}}\right)^{2}\frac{\partial^{2}u}{\partial r^{2}},
\]
with 
\[
\frac{\partial r}{\partial r^{\xi}}=\frac{1}{1-\Psi_{z}\epsilon g},\quad\quad\frac{\partial^{2}r}{\partial(r^{\xi})^{2}}=\frac{-\Psi_{zz}\epsilon g}{(1-\Psi_{z}\epsilon g)^{3}},
\]
\[
\frac{\partial r}{\partial\theta^{\xi}}=\frac{-\Psi\epsilon g_{\theta}}{1-\Psi_{z}\epsilon g},\quad\quad\frac{\partial^{2}r}{\partial(\theta^{\xi})^{2}}=\frac{-\Psi\epsilon g_{\theta\theta}}{1-\Psi_{z}\epsilon g}-\frac{2\Psi\Psi_{z}(\epsilon g_{\theta})^{2}}{(1-\Psi_{z}\epsilon g)^{2}}-\frac{\Psi^{2}\Psi_{zz}\epsilon g(\epsilon g_{\theta})^{2}}{(1-\Psi_{z}\epsilon g)^{3}},
\]
where $\Psi_{z}=\partial\Psi/\partial z$, $\Psi_{zz}=\partial^{2}\Psi/\partial z^{2}$,
$g^{\xi}(\theta^{\xi},\varphi^{\xi},\tau)=g(\theta,\varphi,\tau)$,
and we can deduce the formulas for $\varphi$ in a similar fashion.
Let $\operatorname{ext}\mathbb{B}$ denote the set $\{r>1\}$. By
applying the Hanzawa transformation (\ref{eq:hanzawa transformation})
to (\ref{eq:ED rescaled system xi}), we finally conclude that (\ref{eq:ED rescaled system xi})
is equivalent to the following system: 
\begin{equation}
\begin{cases}
\Delta P=f^{1}[P,g] & \textnormal{in}\,\mathbb{B},\quad\tau>0,\\
-\nabla P+\Delta\mathbf{U}=\mathbf{f}^{2}[\mathbf{U},P,g] & \textnormal{in}\,\mathbb{B},\quad\tau>0,\\
\operatorname{div}\mathbf{U}=f^{3}[\mathbf{U},g] & \textnormal{on}\,\mathbb{S},\quad\tau>0,\\
{}[-P\mathbf{I}+\nabla\mathbf{U}+(\nabla\mathbf{U})^{T}]\mathbf{e}_{r}-\sigma(g+\frac{1}{2}\Delta_{\omega}g)\mathbf{e}_{r}=\mathbf{f}^{4}[\mathbf{U},P,g] & \textnormal{on}\,\mathbb{S},\quad\tau>0,\\
\int_{\mathbb{B}}\mathbf{U}dV=\mathbf{f}^{5}[\mathbf{U},g], & \tau>0,\\
\int_{\mathbb{B}}\left(\mathbf{U}\times(r\sin\theta\cos\varphi,r\sin\theta\sin\varphi,r\cos\theta)\right)dV=\mathbf{f}^{6}[\mathbf{U},g], & \tau>0,\\
\Delta\Phi=f^{7}[\Phi,g] & \textnormal{in}\,\operatorname{ext}\mathbb{B},\quad\tau>0,\\
\Phi-g=f^{8}[g] & \textnormal{on}\,\mathbb{S},\quad\tau>0,\\
\Phi\rightarrow0 & \textnormal{as}\, r\rightarrow\infty,
\end{cases}\label{eq:ED system fixed domain}
\end{equation}
where 
\begin{align*}
f^{1}\left[P,g\right] & =\left[1-\frac{1}{\left(1-\Psi_{z}\epsilon g\right)^{2}}\right]\frac{\partial^{2}P}{\partial r^{2}}-\frac{-\Psi_{zz}\epsilon g}{\left(1-\Psi_{z}\epsilon g\right)^{3}}\frac{\partial P}{\partial r}+\frac{2}{r}\left(1-\frac{r}{r+\Psi\epsilon g}\frac{1}{1-\Psi_{z}\epsilon g}\right)\frac{\partial P}{\partial r}\\
 & +\frac{1}{r^{2}}\left[1-\frac{r^{2}}{\left(r+\Psi\epsilon g\right)^{2}}\right]\Delta_{\omega}P-\frac{1}{\left(r+\Psi\epsilon g\right)^{2}}\left(\frac{\cos\theta}{\sin\theta}\frac{-\Psi\epsilon g_{\theta}}{1-\Psi_{z}\epsilon g}\frac{\partial P}{\partial r}+A^{\theta}[P,g]+\frac{1}{\sin^{2}\theta}A^{\varphi}[P,g]\right),
\end{align*}

\begin{equation}
\Delta_{\omega}=\frac{1}{\sin\theta}\frac{\partial}{\partial\theta}\left(\sin\theta\frac{\partial}{\partial\theta}\right)+\frac{1}{\sin^{2}\theta}\frac{\partial^{2}}{\partial\varphi^{2}},\label{eq:laplacian on sphere}
\end{equation}
\begin{align}
A^{\theta}[P,g] & =\frac{-\Psi\epsilon g_{\theta}}{1-\Psi_{z}\epsilon g}2\frac{\partial^{2}P}{\partial r\partial\theta}+\left[\frac{-\Psi\epsilon g_{\theta\theta}}{1-\Psi_{z}\epsilon g}+\frac{-2\Psi\Psi_{z}\left(\epsilon g_{\theta}\right)^{2}}{\left(1-\Psi_{z}\epsilon g\right)^{2}}+\frac{-\Psi^{2}\Psi_{zz}\epsilon g\left(\epsilon g_{\theta}\right)^{2}}{\left(1-\Psi_{z}\epsilon g\right)^{3}}\right]\frac{\partial P}{\partial r}\label{s23}\\
 & +\left(\frac{-\Psi\epsilon g_{\theta}}{1-\Psi_{z}\epsilon g}\right)^{2}\frac{\partial^{2}P}{\partial r^{2}},\nonumber 
\end{align}
and we have $A^{\varphi}[P,g]$ from replacing $\theta$ by $\varphi$
in (\ref{s23}). In addition, we deduce

\begin{align*}
\mathbf{f}^{2}[\mathbf{U},P,g] & =f^{1}[\mathbf{U},g]+\left(1-\frac{1}{1-\Psi_{z}\epsilon g}\right)\frac{\partial P}{\partial r}\mathbf{e}_{r}\\
 & +\left[\frac{1}{r}\left(1-\frac{r}{r+\Psi\epsilon g}\right)\frac{\partial P}{\partial\theta}-\frac{1}{r+\Psi\epsilon g}\frac{-\Psi\epsilon g_{\theta}}{1-\Psi_{z}\epsilon g}\frac{\partial P}{\partial r}\right]\mathbf{e}_{\theta}\\
 & +\frac{1}{\sin\theta}\left[\frac{1}{r}\left(1-\frac{r}{r+\Psi\epsilon g}\right)\frac{\partial P}{\partial\varphi}-\frac{1}{r+\Psi\epsilon g}\frac{-\Psi\epsilon g_{\varphi}}{1-\Psi_{z}\epsilon g}\frac{\partial P}{\partial r}\right]\mathbf{e}_{\varphi},
\end{align*}

\begin{align*}
f^{3}[\mathbf{U},g] & =2\left(1-\frac{1}{1+\epsilon g}\right)(\mathbf{U}\cdot\mathbf{e}_{r})+\left(1-\frac{1}{1-\Psi_{z}\epsilon g}\right)\frac{\partial\left(\mathbf{U}\cdot\mathbf{e}_{r}\right)}{\partial r}\\
 & +\left(1-\frac{1}{1+\epsilon g}\right)\left[\frac{\cos\theta}{\sin\theta}\left(\mathbf{U}\cdot\mathbf{e}_{\theta}\right)+\frac{\partial\left(\mathbf{U}\cdot\mathbf{e}_{\theta}\right)}{\partial\theta}\right]-\frac{1}{1+\epsilon g}\frac{-\Psi\epsilon g_{\theta}}{1-\Psi_{z}\epsilon g}\frac{\partial\left(\mathbf{U}\cdot\mathbf{e}_{\theta}\right)}{\partial r}\\
 & +\left(1-\frac{1}{1+\epsilon g}\right)\frac{1}{\sin\theta}\frac{\partial\left(\mathbf{U}\cdot\mathbf{e}_{\varphi}\right)}{\partial\varphi}-\frac{1}{1+\epsilon g}\frac{1}{\sin\theta}\frac{-\Psi\epsilon g_{\varphi}}{1-\Psi_{z}\epsilon g}\frac{\partial\left(\mathbf{U}\cdot\mathbf{e}_{\varphi}\right)}{\partial r},
\end{align*}
\[
\mathbf{f}^{4}[\mathbf{U},P,g]=\left(-P\mathbf{I}+\mathbf{B}_{1}+\mathbf{B}_{1}^{T}\right)\left(\mathbf{e}_{r}-\mathbf{n}^{\xi}\right)+\left(\mathbf{B}_{2}+\mathbf{B}_{2}^{T}\right)\mathbf{e}_{r}+\sigma(1-\kappa^{\xi})\epsilon^{-1}\left(\mathbf{\mathbf{n}}^{\xi}-\mathbf{e}_{r}\right)-\sigma f^{4;\kappa}[g]\mathbf{e}_{r},
\]
where

\[
\mathbf{B}_{1}=\frac{1}{1-\Psi_{z}\epsilon g}\frac{\partial\mathbf{U}}{\partial r}\mathbf{e}_{r}+\frac{1}{1+\epsilon g}\left(\frac{\partial\mathbf{U}}{\partial\theta}\mathbf{e}_{\theta}+\frac{1}{\sin\theta}\frac{\partial\mathbf{U}}{\partial\varphi}\mathbf{e}_{\varphi}\right)+\frac{1}{1+\epsilon g}\frac{\partial\mathbf{U}}{\partial r}\left(\frac{-\Psi\epsilon g_{\theta}}{1-\Psi_{z}\epsilon g}\mathbf{e}_{\theta}+\frac{1}{\sin\theta}\frac{-\Psi\epsilon g_{\epsilon}}{1-\Psi_{z}\epsilon g}\mathbf{e}_{\varphi}\right),
\]

\begin{align*}
\mathbf{B}_{2} & =\left(1-\frac{1}{1-\Psi_{z}\epsilon g}\right)\frac{\partial\mathbf{U}}{\partial r}\mathbf{e}_{r}\\
 & +\left(1-\frac{1}{1+\epsilon g}\right)\left(\frac{\partial\mathbf{U}}{\partial\theta}\mathbf{e}_{\theta}+\frac{1}{\sin\theta}\frac{\partial\mathbf{U}}{\partial\varphi}\mathbf{e}_{\varphi}\right)-\frac{1}{1+\epsilon g}\frac{\partial\mathbf{U}}{\partial r}\left(\frac{-\Psi\epsilon g_{\theta}}{1-\Psi_{z}\epsilon g}\mathbf{e}_{\theta}+\frac{1}{\sin\theta}\frac{-\Psi\epsilon g_{\epsilon}}{1-\Psi_{z}\epsilon g}\mathbf{e}_{\varphi}\right),
\end{align*}
and from Theorem 8.1 in \cite{Friedman:2001vk} we have

\[
\kappa^{\xi}=1-\epsilon(g+\frac{1}{2}\Delta_{\omega}g+f^{4;\kappa}[g]),\quad\quad f^{4;\kappa}[g]=O(\epsilon).
\]
We also arrive at
\[
\mathbf{f}^{5}[\mathbf{U},g]=\int_{\mathbb{B}}\mathbf{U}\left[1-\left(1+\Psi\epsilon g/r\right)^{2}\left(1-\Psi_{z}\epsilon g\right)\right]dV,
\]

\[
\mathbf{f}^{6}[\mathbf{U},g]=\int_{\mathbb{B}}\mathbf{U}\times\left(r\sin\theta\cos\varphi,r\sin\theta\sin\varphi,r\cos\theta\right)\left[1-\left(1+\Psi\epsilon g/r\right)^{3}\left(1-\Psi_{z}\epsilon g\right)\right]dV,
\]

\[
f^{7}[\Phi,g]=f^{1}[\Phi,g],
\]

\[
f^{8}[g]=\frac{-\epsilon g^{2}}{1+\epsilon g}.
\]
We emphasize that the order of the right-hand sides of (\ref{eq:ED system fixed domain})
is $O(\epsilon)$, which will be used in our proofs.

We now consider the kinematic condition (\ref{eq:ED kinematic condition primary-1}).
Using the self-similar variables (\ref{eq:self-similar variables solution})
and (\ref{eq:self-similar variables}), we deduce the following equation
in terms of the spherical coordinate system $(r^{\xi},\theta^{\xi},\varphi^{\xi})$:
\begin{equation}
v_{\mathbf{n}}=\frac{d}{dt}\left[R(t)(1+\epsilon g^{\xi})\mathbf{e}_{(r^{\xi})}\right]\cdot\mathbf{n}^{\xi}=\frac{1}{2(1+\epsilon t_{0}Y_{0,0})e^{-\tau}}\left(-1+\frac{\partial\epsilon g^{\xi}}{\partial\tau}-\epsilon g^{\xi}\right)\mathbf{e}_{(r^{\xi})}\cdot\mathbf{n}^{\xi}\quad\textnormal{on}\,\Gamma^{\xi}(\tau).\label{s211}
\end{equation}
In addition, one can easily see that 
\begin{equation}
\mathbf{n}^{\xi}=\frac{1+\epsilon g^{\xi}}{\varrho^{\xi}}\mathbf{e}_{(r^{\xi})}+\frac{-\epsilon g_{(\theta^{\xi})}^{\xi}}{\varrho^{\xi}}\mathbf{e}_{(\theta^{\xi})}+\frac{-\epsilon g_{(\varphi^{\xi})}^{\xi}}{\varrho^{\xi}\sin\theta^{\xi}}\mathbf{e}_{(\varphi^{\xi})}\quad\textnormal{with}\quad\varrho^{\xi}=\left[(1+\epsilon g^{\xi})^{2}+(\epsilon g_{(\theta^{\xi})}^{\xi})^{2}+(\epsilon g_{(\varphi^{\xi})}^{\xi}/\sin\theta^{\xi})^{2}\right]^{1/2}\label{s212}
\end{equation}
holds. Substituting (\ref{s212}) into (\ref{s211}) we can rewrite
(\ref{eq:ED kinematic condition primary-1}) in the form 
\begin{align}
\frac{\partial g^{\xi}}{\partial\tau}-g^{\xi} & =2(1+\epsilon t_{0}Y_{0,0})e^{-\tau}\left(\mathbf{U}^{\xi}\cdot\mathbf{e}_{(r^{\xi})}-\mathbf{U}^{\xi}\cdot\mathbf{e}_{(\theta^{\xi})}\frac{\epsilon g_{(\theta^{\xi})}^{\xi}}{1+\epsilon g^{\xi}}-\mathbf{U}^{\xi}\cdot\mathbf{e}_{(\varphi^{\xi})}\frac{\epsilon g_{(\varphi^{\xi})}^{\xi}}{1+\epsilon g^{\xi}}\frac{1}{\sin\theta^{\xi}}\right)\label{s213}\\
 & +\frac{1}{\epsilon}\left[1-\frac{1+\epsilon t_{0}Y_{0,0}}{\left(1+\epsilon g^{\xi}\right)^{2}}\right]\nonumber \\
 & +\left(1+\epsilon t_{0}Y_{0,0}\right)\left(\frac{\partial\Phi^{\xi}}{\partial r^{\xi}}-\frac{1}{1+\epsilon g^{\xi}}\frac{\partial\Phi^{\xi}}{\partial\theta^{\xi}}\frac{\epsilon g_{(\theta^{\xi})}^{\xi}}{1+\epsilon g^{\xi}}-\frac{1}{1+\epsilon g^{\xi}}\frac{\partial\Phi^{\xi}}{\partial\varphi^{\xi}}\frac{\epsilon g_{(\varphi^{\xi})}^{\xi}}{1+\epsilon g^{\xi}}\frac{1}{\sin^{2}\theta^{\xi}}\right)\quad\textnormal{on}\,\Gamma^{\xi}(\tau).\nonumber 
\end{align}
Therefore, by means of the Hanzawa transformation (\ref{eq:hanzawa transformation}),
we arrive from (\ref{s213}) at the following equation for the kinematic
condition (\ref{eq:ED kinematic condition primary-1}): 
\begin{equation}
\frac{\partial g}{\partial\tau}-3g+t_{0}Y_{0,0}-2e^{-\tau}(\mathbf{U}\cdot\mathbf{e}_{r})-\frac{\partial\Phi}{\partial r}=e^{-\tau}f^{9}[\mathbf{U},g,t_{0}]+f^{10}[\Phi,g,t_{0}]\quad\textnormal{on}\,\mathbb{S},\quad\tau>0,\label{eq:ED kinematic condition fixed domain}
\end{equation}
where 
\[
f^{9}[\mathbf{U},g,t_{0}]=2(1+\epsilon t_{0}Y_{0,0})\left(-\mathbf{U}\cdot\mathbf{e}_{\theta}\frac{\epsilon g_{\theta}}{1+\epsilon g}-\mathbf{U}\cdot\mathbf{e}_{\varphi}\frac{\epsilon g_{\varphi}}{1+\epsilon g}\frac{1}{\sin\theta}\right)+2\epsilon t_{0}Y_{0,0}\mathbf{U}\cdot\mathbf{e}_{r},
\]
\begin{align*}
f^{10}[\Phi,g,t_{0}] & =(1+\epsilon t_{0}Y_{0,0})\left[-\frac{3\epsilon g^{2}+2\epsilon^{2}g^{3}}{\left(1+\epsilon g\right)^{2}}-\frac{1}{1+\epsilon g}\left(\frac{\partial\Phi}{\partial\theta}+\frac{-\Psi\epsilon g_{\theta}}{1-\Psi_{z}\epsilon g}\frac{\partial\Phi}{\partial r}\right)\frac{\epsilon g_{\theta}}{1+\epsilon g}\right.\\
 & \left.-\frac{1}{1+\epsilon g}\left(\frac{\partial\Phi}{\partial\varphi}+\frac{-\Psi\epsilon g_{\varphi}}{1-\Psi_{z}\epsilon g}\frac{\partial\Phi}{\partial r}\right)\frac{\epsilon g_{\varphi}}{1+\epsilon g}\frac{1}{\sin^{2}\theta}\right]+\epsilon t_{0}Y_{0,0}\frac{\partial\Phi}{\partial r}+2\epsilon t_{0}Y_{0,0}g.
\end{align*}
The order of $f^{9}[\mathbf{U},g,t_{0}]$ and $f^{10}[\Phi,g,t_{0}]$
is $O(\epsilon)$.

We finally consider the initial condition (\ref{eq:ED initial condition primary-1}).
Assume that the initial shape
\[
\Gamma^{0}=\left(1+\epsilon g^{0}(\theta^{\prime},\varphi^{\prime})\right)\mathbf{e}_{(r^{\prime})}
\]
is a perturbation of the self-similar solution (\ref{eq:ED self-similar solution}),
where $(r^{\prime},\theta^{\prime},\varphi^{\prime})$ is the spherical
coordinate system for the three-dimensional $\mathbf{x}$-space and
$g^{0}(\theta^{\prime},\varphi^{\prime})$ is given. Then, by the
self-similar variables (\ref{eq:self-similar variables solution})
and (\ref{eq:self-similar variables}), we are led to the following
relation:
\begin{equation}
\left(1+\epsilon g^{0}(\theta^{\prime},\varphi^{\prime})\right)\mathbf{e}_{(r^{\prime})}=\left(1+\epsilon g^{\xi}(\theta^{\xi},\varphi^{\xi},0)\right)\mathbf{e}_{(r^{\xi})}+\epsilon\mathbf{x}_{0}.\label{in1}
\end{equation}
Writing $\epsilon\mathbf{x}_{0}=\epsilon(x_{01},x_{02},x_{03})$,
we deduce from (\ref{in1}) that
\begin{equation}
\begin{cases}
\left(1+\epsilon g^{\xi}(\theta^{\xi},\varphi^{\xi},0)\right)\sin\theta^{\xi}\cos\varphi^{\xi}=\left(1+\epsilon g^{0}(\theta^{\prime},\varphi^{\prime})\right)\sin\theta^{\prime}\cos\varphi^{\prime}-\epsilon x_{01},\\
\left(1+\epsilon g^{\xi}(\theta^{\xi},\varphi^{\xi},0)\right)\sin\theta^{\xi}\sin\varphi^{\xi}=\left(1+\epsilon g^{0}(\theta^{\prime},\varphi^{\prime})\right)\sin\theta^{\prime}\sin\varphi^{\prime}-\epsilon x_{02},\\
\left(1+\epsilon g^{\xi}(\theta^{\xi},\varphi^{\xi},0)\right)\cos\theta^{\xi}=\left(1+\epsilon g^{0}(\theta^{\prime},\varphi^{\prime})\right)\cos\theta^{\prime}-\epsilon x_{03},
\end{cases}\label{in2}
\end{equation}
so that one can write 
\[
\left(1+\epsilon g^{\xi}(\theta^{\xi},\varphi^{\xi},0)\right)^{2}=f^{angle}(\theta^{\prime},\varphi^{\prime})
\]
with 
\[
f^{angle}(\theta^{\prime},\varphi^{\prime})=\left(1+\epsilon g^{0}(\theta^{\prime},\varphi^{\prime})\right)^{2}+\epsilon^{2}|\mathbf{x}_{0}|^{2}-2\epsilon\left(1+\epsilon g^{0}(\theta^{\prime},\varphi^{\prime})\right)\left(x_{01}\sin\theta^{\prime}\cos\varphi^{\prime}+x_{02}\sin\theta^{\prime}\sin\varphi^{\prime}+x_{03}\cos\theta^{\prime}\right).
\]
Accordingly, it follows from (\ref{in2}) that $\theta^{\xi}$ and
$\varphi^{\xi}$ are functions of $(\theta^{\prime},\varphi^{\prime})$:
\begin{align*}
\theta^{\xi} & =\cos^{-1}\left[\frac{\left(1+\epsilon g^{0}(\theta^{\prime},\varphi^{\prime})\right)\cos\theta^{\prime}-\epsilon x_{03}}{\sqrt{f^{angle}(\theta^{\prime},\varphi^{\prime})}}\right]=\theta^{\prime}+O(\epsilon),\\
\varphi^{\xi} & =\cos^{-1}\left[\frac{\left(1+\epsilon g^{0}(\theta^{\prime},\varphi^{\prime})\right)\sin\theta^{\prime}\cos\varphi^{\prime}-\epsilon x_{01}}{\sin\theta^{\xi}\sqrt{f^{angle}(\theta^{\prime},\varphi^{\prime})}}\right]=\varphi^{\prime}+O(\epsilon).
\end{align*}
As a result, since there is no change of angles in the Hanzawa transformation
(\ref{eq:hanzawa transformation}) and $g^{\xi}(\theta^{\xi},\varphi^{\xi},0)=g(\theta,\varphi,0)$
holds, we rewrite $g(\theta,\varphi,0)$ in the form 
\[
g(\theta,\varphi,0)-g^{0}(\theta,\varphi)+\left(x_{01}\sin\theta\cos\varphi+x_{02}\sin\theta\sin\varphi+x_{03}\cos\theta\right)=f^{11}[\mathbf{x}_{0}],
\]
where the order of $f^{11}[\mathbf{x}_{0}]$ is $O(\epsilon)$. Moreover,
by using 
\[
Y_{1,-1}=\sqrt{\frac{3}{8\pi}}\sin\theta e^{-i\varphi},\quad Y_{1,0}=\sqrt{\frac{3}{4\pi}}\cos\theta,\quad Y_{1,1}=-\sqrt{\frac{3}{8\pi}}\sin\theta e^{i\varphi},
\]
the initial condition (\ref{eq:ED initial condition primary-1}) is
reduced to 
\begin{equation}
g(\theta,\varphi,0)-g^{0}(\theta,\varphi)+\sum_{|m|\leq1}b_{1,m}Y_{1,m}=f^{11}[\mathbf{x}_{0}],\label{eq:ED initial condition fixed domain}
\end{equation}
where 
\begin{equation}
\left(\begin{array}{c}
x_{01}\\
x_{02}\\
x_{03}
\end{array}\right)=\sqrt{\frac{3}{8\pi}}\left(\begin{array}{ccc}
1 & 0 & -1\\
-i & 0 & -i\\
0 & \sqrt{2} & 0
\end{array}\right)\left(\begin{array}{c}
b_{1,-1}\\
b_{1,0}\\
b_{1,1}
\end{array}\right).\label{eq:x0 b1 relation}
\end{equation}

\subsection{The decomposition of the solution}

We decompose the fixed-domain problem (\ref{eq:ED system fixed domain}),
(\ref{eq:ED kinematic condition fixed domain}), (\ref{eq:ED initial condition fixed domain})
into the ellipsoidal part and the evolution part. We will use the
super-indexes $\mathfrak{E}$ and $\mathfrak{T}$ to denote the unknowns
for such ellipsoidal and evolution problems, respectively. Accordingly,
the decomposition is 
\begin{equation}
\left(P,\mathbf{U},\Phi,g\right)=(P^{\mathfrak{E}},\mathbf{U}^{\mathfrak{E}},\Phi^{\mathfrak{E}},g^{\mathfrak{E}})+(P^{\mathfrak{T}},\mathbf{U}^{\mathfrak{T}},\Phi^{\mathfrak{T}},g^{\mathfrak{T}}).\label{eq:decomposition}
\end{equation}
We define (\ref{eq:decomposition}) more precisely. The ellipsoidal
part satisfies the following overdetermined problem: 
\begin{equation}
\begin{cases}
\Delta P^{\mathfrak{E}}=f^{1}[P^{\mathfrak{E}},g^{\mathfrak{E}}] & \textnormal{in}\,\mathbb{B},\\
-\nabla P^{\mathfrak{E}}+\Delta\mathbf{U}^{\mathfrak{E}}=\mathbf{f}^{2}[\mathbf{U}^{\mathfrak{E}},P^{\mathfrak{E}},g^{\mathfrak{E}}] & \textnormal{in}\,\mathbb{B},\\
\operatorname{div}\mathbf{U}^{\mathfrak{E}}=f^{3}[\mathbf{U}^{\mathfrak{E}},g^{\mathfrak{E}}] & \textnormal{on}\,\mathbb{S},\\
{}[-P^{\mathfrak{E}}\mathbf{I}+\nabla\mathbf{U}^{\mathfrak{E}}+(\nabla\mathbf{U}^{\mathfrak{E}})^{T}]\mathbf{e}_{r}-\sigma(g^{\mathfrak{E}}+\frac{1}{2}\Delta_{\omega}g^{\mathfrak{E}})\mathbf{e}_{r}=\mathbf{f}^{4}[\mathbf{U}^{\mathfrak{E}},P^{\mathfrak{E}},g^{\mathfrak{E}}] & \textnormal{on}\,\mathbb{S},\\
\int_{\mathbb{B}}\mathbf{U}^{\mathfrak{E}}dV=\mathbf{f}^{5}[\mathbf{U}^{\mathfrak{E}},g^{\mathfrak{E}}],\\
\int_{\mathbb{B}}\left(\mathbf{U}^{\mathfrak{E}}\times(r\sin\theta\cos\varphi,r\sin\theta\sin\varphi,r\cos\theta)\right)dV=\mathbf{f}^{6}[\mathbf{U}^{\mathfrak{E}},g^{\mathfrak{E}}],
\end{cases}\label{eq:dce stokes system}
\end{equation}
\begin{equation}
\begin{cases}
\Delta\Phi^{\mathfrak{E}}=f^{7}[\Phi^{\mathfrak{E}},g^{\mathfrak{E}}] & \textnormal{in}\,\operatorname{ext}\mathbb{B},\\
\Phi^{\mathfrak{E}}-g^{\mathfrak{E}}=f^{8}[g^{\mathfrak{E}}] & \textnormal{on}\,\mathbb{S},\\
\Phi^{\mathfrak{E}}\rightarrow0 & \textnormal{as}\, r\rightarrow\infty,
\end{cases}\label{eq:dce vapor}
\end{equation}
\texttt{ 
\begin{equation}
-3g^{\dce}+t_{0}Y_{0,0}-\frac{\partial\Phi^{\dce}}{\partial r}=f^{10}[\Phi^{\dce},g^{\dce},t_{0}]\quad\textnormal{on}\,\mathbb{S},\label{eq:dce boundary condition}
\end{equation}
}where $t_{0}$ is a part of the solution to (\ref{eq:dce stokes system})-(\ref{eq:dce boundary condition}).
On the other hand, the evolution part satisfies the following evolution
problem: 
\begin{equation}
\begin{cases}
\Delta P^{\mathfrak{T}}=f^{1}[P,g]-f^{1}[P^{\mathfrak{E}},g^{\mathfrak{E}}] & \textnormal{in}\,\mathbb{B},\quad\tau>0,\\
-\nabla P^{\mathfrak{T}}+\Delta\mathbf{U}^{\mathfrak{T}}=\mathbf{f}^{2}[\mathbf{U},P,g]-\mathbf{f}^{2}[\mathbf{U}^{\mathfrak{E}},P^{\mathfrak{E}},g^{\mathfrak{E}}] & \textnormal{in}\,\mathbb{B},\quad\tau>0,\\
\operatorname{div}\mathbf{U}^{\mathfrak{T}}=f^{3}[\mathbf{U},g]-f^{3}[\mathbf{U}^{\mathfrak{E}},g^{\mathfrak{E}}] & \textnormal{on}\,\mathbb{S},\quad\tau>0,\\
{}[-P^{\mathfrak{T}}\mathbf{I}+\nabla\mathbf{U}^{\mathfrak{T}}+(\nabla\mathbf{U}^{\mathfrak{T}})^{T}]\mathbf{e}_{r}-\sigma(g^{\mathfrak{T}}+\frac{1}{2}\Delta_{\omega}g^{\mathfrak{T}})\mathbf{e}_{r}=\mathbf{f}^{4}[\mathbf{U},P,g]-\mathbf{f}^{4}[\mathbf{U}^{\mathfrak{E}},P^{\mathfrak{E}},g^{\mathfrak{E}}] & \textnormal{on}\,\mathbb{S},\quad\tau>0,\\
\int_{\mathbb{B}}\mathbf{U}^{\mathfrak{T}}dV=\mathbf{f}^{5}[\mathbf{U},g]-\mathbf{f}^{5}[\mathbf{U}^{\mathfrak{E}},g^{\mathfrak{E}}], & \tau>0,\\
\int_{\mathbb{B}}\left(\mathbf{U}^{\mathfrak{T}}\times(r\sin\theta\cos\varphi,r\sin\theta\sin\varphi,r\cos\theta)\right)dV=\mathbf{f}^{6}[\mathbf{U},g]-\mathbf{f}^{6}[\mathbf{U}^{\mathfrak{E}},g^{\mathfrak{E}}], & \tau>0,\\
\Delta\Phi^{\mathfrak{T}}=f^{7}[\Phi,g]-f^{7}[\Phi^{\mathfrak{E}},g^{\mathfrak{E}}] & \textnormal{in}\,\operatorname{ext}\mathbb{B},\quad\tau>0,\\
\Phi^{\mathfrak{T}}-g^{\mathfrak{T}}=f^{8}[g]-f^{8}[g^{\mathfrak{E}}] & \textnormal{on}\,\mathbb{S},\quad\tau>0,\\
\Phi^{\mathfrak{T}}\rightarrow0 & \textnormal{as}\, r\rightarrow\infty,
\end{cases}\label{eq:dct system}
\end{equation}
\texttt{ 
\begin{equation}
\frac{\partial g^{\dct}}{\partial\tau}-3g^{\mathfrak{T}}-2e^{-\tau}(\mathbf{U}\cdot\mathbf{e}_{r})-\frac{\partial\Phi^{\mathfrak{T}}}{\partial r}=e^{-\tau}f^{9}[\mathbf{U},g,t_{0}]+f^{10}[\Phi,g,t_{0}]-f^{10}[\Phi^{\mathfrak{E}},g^{\mathfrak{E}},t_{0}]\quad\textnormal{on}\,\mathbb{S},\quad\tau>0,\label{eq:dct kinematic condition}
\end{equation}
\begin{equation}
g^{\mathfrak{T}}(\theta,\varphi,0)-g^{0}+\sum_{|m|\leq1}b_{1,m}Y_{1,m}+g^{\mathfrak{E}}=f^{11}[\mathbf{x}_{0}].\label{eq:dct initial condition}
\end{equation}
}Here, the solution to (\ref{eq:dce stokes system})-(\ref{eq:dce boundary condition})
enters (\ref{eq:dct system}) and (\ref{eq:dct kinematic condition})
so that the right-hand sides of (\ref{eq:dct system}) and (\ref{eq:dct kinematic condition})
can be regarded as a perturbation of the ellipsoidal part. Moreover,
$U^{\mathfrak{E}}$ and $g^{\mathfrak{E}}$ appear linearly on the
left-hand sides of (\ref{eq:dct kinematic condition}) and (\ref{eq:dct initial condition}).
This special structure of the system (\ref{eq:dct system})-(\ref{eq:dct initial condition})
will allow us to compute approximations to the eigenvalues of a certain
linearized system.

\section{\label{sec:Main theorems for the reformulated problem}The main theorems
for the reformulated problem}

The aim of this section is to state the main theorems for the reformulated
problem (\ref{eq:dce stokes system})-(\ref{eq:dct initial condition}).
Using these main theorems and reversing the procedure in Section \ref{sec:Reformulation of the problem},
we can readily prove Theorem \ref{thm:stable ellipsoidal collapse to a point}.

Our first main theorem is the following:
\begin{thm}[Existence of ellipsoids]
\label{thm:existence of ellipsoids}The set $\{r=1+\epsilon g^{\mathfrak{E}}\}$
of the overdetermined problem (\ref{eq:dce stokes system})-(\ref{eq:dce boundary condition})
must be an ellipsoid. Moreover, assume that $g_{0,0}^{\mathfrak{E}}$
and $\{g_{2,m}^{\mathfrak{E}}\}_{|m|\leq2}$ of $g^{\mathfrak{E}}=\sum_{l,m}g_{l,m}^{\mathfrak{E}}Y_{l,m}$
are given complex constants such that 
\[
g_{0,0}^{\dce},g_{2,0}^{\dce}\in\mathbb{R},\quad g_{2,-1}^{\dce}=-\overline{g_{2,1}^{\dce}}\in\mathbb{C},\quad g_{2,-2}^{\dce}=\overline{g_{2,2}^{\dce}}\in\mathbb{C},
\]
so that $g^{\dce}$ is real. Then for a sufficiently small $\epsilon>0$,
(\ref{eq:dce stokes system})-(\ref{eq:dce boundary condition}) has
a unique solution $(P^{\mathfrak{E}},\mathbf{U}^{\mathfrak{E}},\Phi^{\mathfrak{E}},g^{\mathfrak{E}},t_{0})$
such that 
\[
\{r=1+\epsilon g^{\mathfrak{E}}\}=\{r=1+\epsilon g_{0,0}^{\dce}Y_{0,0}+\epsilon\sum_{|m|\leq2}g_{2,m}^{\dce}Y_{2,m}+O(\epsilon^{2})\}
\]
is an ellipsoid whose semi-axes are of lengths $\mathfrak{a}$, $\mathfrak{b}$,
and $\mathfrak{c}$ and $t_{0}$ satisfies 
\begin{equation}
1+\epsilon t_{0}Y_{0,0}=\frac{\mathfrak{abc}}{2}\int_{0}^{\infty}\frac{ds}{\sqrt{(\mathfrak{a}^{2}+s)(\mathfrak{b}^{2}+s)(\mathfrak{c}^{2}+s)}}.\label{eq:ellipsoid t0}
\end{equation}
The following holds for any integer $n\geq4$: 
\begin{equation}
\Vert P^{\mathfrak{E}}\Vert_{H^{n-2}(\mathbb{B})}+\Vert\mathbf{U}^{\mathfrak{E}}\Vert_{H^{n-1}(\mathbb{B})}+\Vert\Phi^{\mathfrak{E}}\Vert_{H^{n}(\operatorname{ext}\mathbb{B};w)}+\Vert g^{\mathfrak{E}}\Vert_{H^{n-\frac{1}{2}}(\mathbb{S})}+|t_{0}|\leq C(|g_{0,0}^{\dce}|+\sum_{|m|\leq2}|g_{2,m}^{\dce}|),\label{eq:estimate1}
\end{equation}
where $C$ is some constant.\end{thm}
\begin{proof}
We will prove Theorem \ref{thm:existence of ellipsoids} in Section
\ref{sec:Existence of ellipsoids}.
\end{proof}
By Theorem \ref{thm:existence of ellipsoids}, it immediately follows
that the leading-order terms of the solution to (\ref{eq:dce stokes system})-(\ref{eq:dce boundary condition})
consist of $g_{0,0}^{\dce}$ and $\left\{ g_{2,m}^{\dce}\right\} _{|m|\leq2}$.
As a result, we can rewrite (\ref{eq:dct kinematic condition}) and
(\ref{eq:dct initial condition}) in the form
\begin{align}
\frac{\partial g^{\dct}}{\partial\tau}-3g^{\mathfrak{T}}-\frac{\partial\Phi^{\mathfrak{T}}}{\partial r}-2e^{-\tau}(\mathbf{U}^{\mathfrak{T}}\cdot\mathbf{e}_{r}-\sigma\frac{10}{19}\sum_{|m|\leq2}g_{2,m}^{\mathfrak{E}}Y_{2,m}) & =e^{-\tau}f^{9}[\mathbf{U},g,t_{0}]+f^{10}[\Phi,g,t_{0}]-f^{10}[\Phi^{\mathfrak{E}},g^{\mathfrak{E}},t_{0}]\label{eq:dct re kinematic}\\
 & \underset{O(\epsilon)}{\underbrace{+2e^{-\tau}(\mathbf{U}^{\mathfrak{E}}\cdot\mathbf{e}_{r}+\sigma\frac{10}{19}\sum_{|m|\leq2}g_{2,m}^{\mathfrak{E}}Y_{2,m})}},\nonumber 
\end{align}
 
\begin{equation}
g^{\mathfrak{T}}(\theta,\varphi,0)-g^{0}+\sum_{|n|\leq1}b_{1,n}Y_{1,n}+g_{0,0}^{\mathfrak{E}}Y_{0,0}+\sum_{|m|\leq2}g_{2,m}^{\mathfrak{E}}Y_{2,m}=f^{11}[\mathbf{x}_{0}]\underset{O(\epsilon)}{\underbrace{-g^{\dce}+g_{0,0}^{\mathfrak{E}}Y_{0,0}+\sum_{|m|\leq2}g_{2,m}^{\mathfrak{E}}Y_{2,m}}}.\label{eq:dct re initial}
\end{equation}
We have introduced the term $\sigma\frac{10}{19}\sum_{|m|\leq2}g_{2,m}^{\mathfrak{E}}Y_{2,m}$
in (\ref{eq:dct re kinematic}) because it corresponds to the linearization
of $\mathbf{U}^{\mathfrak{E}}\cdot\mathbf{e}_{r}$, which will be
apparent in Section \ref{sec:the ODE system}. Consequently, the order
of the right-hand side of (\ref{eq:dct re kinematic}) is $O(\epsilon)$.
Likewise, the order of the right-hand side of (\ref{eq:dct re initial})
is $O(\epsilon)$.

We introduce the function spaces $\mathscr{X}$ and $\mathscr{Y}$.
Define the following norms (see \cite{Friedman:2001vk}): 
\[
\Vert G\Vert_{s,\mathbb{B}}=\left[\int_{0}^{\infty}e^{2\lambda_{0}\tau}\left(\Vert G(\cdot,\tau)\Vert_{H^{s}(\mathbb{B})}^{2}+\Vert G_{\tau}(\cdot,\tau)\Vert_{H^{s-1}(\mathbb{B})}^{2}+\Vert G_{\tau\tau}(\cdot,\tau)\Vert_{H^{s-2}(\mathbb{B})}^{2}\right)d\tau\right]^{1/2},
\]
\[
\Vert\mathbf{G}\Vert_{s,\mathbb{B}}=\left[\int_{0}^{\infty}e^{2\lambda_{0}\tau}\left(\|\mathbf{G}(\cdot,\tau)\|_{H^{s}(\mathbb{B})}^{2}+\|\mathbf{G}_{\tau}(\cdot,\tau)\|_{H^{s-1}(\mathbb{B})}^{2}+\|\mathbf{G}_{\tau\tau}(\cdot,\tau)\|_{H^{s-2}(\mathbb{B})}^{2}\right)d\tau\right]^{1/2},
\]
\[
\Vert g\Vert_{s,\mathbb{S}}=\left[\int_{0}^{\infty}e^{2\lambda_{0}\tau}\left(\Vert g(\cdot,\tau)\Vert_{H^{s}(\mathbb{S})}^{2}+\Vert g_{\tau}(\cdot,\tau)\Vert_{H^{s-1}(\mathbb{S})}^{2}+\Vert g_{\tau\tau}(\cdot,\tau)\Vert_{H^{s-2}(\mathbb{S})}^{2}\right)d\tau\right]^{1/2},
\]
for $s\geq2$. The number $\lambda_{0}$ is positive and will be chosen
appropriately. We can define $\Vert G\|_{s,\operatorname{ext}\mathbb{B};w}$
in a similar way, but introducing the weighted Sobolev space $H^{s}(\operatorname{ext}\mathbb{B};w)$.
Finally, we denote by $\mathscr{X}$ the Banach space as the set of
vectors $\mathsf{x}$ defined by 
\[
\mathsf{x}=\left(P^{\mathfrak{T}},\mathbf{U}^{\mathfrak{T}},\Phi^{\mathfrak{T}},g^{\mathfrak{T}},g_{0,0}^{\mathfrak{E}}\in\mathbb{R},\left\{ b_{1,n}\in\mathbb{C}\right\} _{\left\vert n\right\vert \leq1},\left\{ g_{2,m}^{\mathfrak{E}}\in\mathbb{C}\right\} _{\left\vert m\right\vert \leq2}\right)
\]
such that 
\[
\Vert\mathsf{x}\Vert_{\mathscr{X}}=\Vert P^{\mathfrak{T}}\Vert_{5,\mathbb{B}}+\Vert\mathbf{U}^{\mathfrak{T}}\Vert_{6,\mathbb{B}}+\Vert\Phi^{\mathfrak{T}}\Vert_{7,\operatorname{ext}\mathbb{B};w}+\Vert g^{\mathfrak{T}}\Vert_{\frac{13}{2},\mathbb{S}}+|g_{0,0}^{\mathfrak{E}}|+\sum_{\left\vert n\right\vert \leq1}|b_{1,n}|+\sum_{\left\vert m\right\vert \leq2}|g_{2,m}^{\mathfrak{E}}|<\infty.
\]
Moreover, Let $\mathscr{Y}$ be the Banach space of elements 
\[
\mathsf{y}=(F^{1;\mathfrak{T}},\mathbf{F}^{2;\mathfrak{T}},F^{3;\mathfrak{T}},\mathbf{F}^{4;\mathfrak{T}},\mathbf{F}^{5;\mathfrak{T}},\mathbf{F}^{6;\mathfrak{T}},F^{7;\mathfrak{T}},F^{8;\mathfrak{T}},F^{9,10;\mathfrak{T}},F^{11;\mathfrak{T}})
\]
such that 
\begin{align*}
\Vert\mathsf{y}\Vert_{\mathscr{Y}} & =\Vert F^{1;\mathfrak{T}}\Vert_{3,\mathbb{B}}+\Vert\mathbf{F}^{2;\mathfrak{T}}\Vert_{4,\mathbb{B}}+\Vert F^{3;\mathfrak{T}}\Vert_{\frac{9}{2},\mathbb{B}}+\Vert\mathbf{F}^{4;\mathfrak{T}}\Vert_{\frac{9}{2},\mathbb{S}}\\
 & +\left(\int_{0}^{\infty}e^{2\lambda_{0}\tau}\sum_{k=0}^{2}|D_{\tau}^{k}\mathbf{F}^{5;\mathfrak{T}}|^{2}d\tau\right)^{1/2}+\left(\int_{0}^{\infty}e^{2\lambda_{0}\tau}\sum_{k=0}^{2}|D_{\tau}^{k}\mathbf{F}^{6;\mathfrak{T}}|^{2}d\tau\right)^{1/2}\\
 & +\Vert F^{7;\mathfrak{T}}\Vert_{5,\operatorname{ext}\mathbb{B};w}+\Vert F^{8;\mathfrak{T}}\Vert_{\frac{13}{2},\mathbb{S}}+\Vert F^{9,10;\mathfrak{T}}\Vert_{\frac{11}{2},\mathbb{S}}+\Vert F^{11;\mathfrak{T}}\Vert_{H^{6}(\mathbb{S})}<\infty.
\end{align*}

We now introduce the linear operator $\mathcal{L}$ defined by the
left-hand sides of the system (\ref{eq:dct system}), (\ref{eq:dct re kinematic}),
(\ref{eq:dct re initial}) and the nonlinear operator $\mathcal{Y}$
defined by the right-hand sides of (\ref{eq:dct system}), (\ref{eq:dct re kinematic}),
(\ref{eq:dct re initial}). By Theorem \ref{thm:existence of ellipsoids},
$\mathcal{L}$ and $\mathcal{Y}$ can be regarded as operators defined
on the Banach space $\mathscr{X}$. In addition, $\mathcal{Y}[\mathsf{x}]$
is a perturbation of the ellipsoidal solution that comes from Theorem
\ref{thm:existence of ellipsoids}. By using these facts, one can
easily check that $\mathcal{Y}\colon\mathscr{X}\rightarrow\mathscr{Y}$.
Therefore the nonlinear problem (\ref{eq:dct system}), (\ref{eq:dct re kinematic}),
(\ref{eq:dct re initial}) can be compactly written in the form 
\[
\mathcal{L}[\mathsf{x}]=\mathcal{Y}[\mathsf{x}],\quad\mathcal{L}\colon\mathscr{X}\rightarrow\mathscr{Y},\quad\mathcal{Y}\colon\mathscr{X}\rightarrow\mathscr{Y}.
\]

Consider first the auxiliary linear evolution problem 
\begin{equation}
\mathcal{L}[\mathsf{x}]=\mathsf{y,\quad y}\in\mathscr{Y}.\label{eq:aux linear problem}
\end{equation}
Then the following theorem holds for (\ref{eq:aux linear problem}):
\begin{thm}[The linear evolution problem]
\label{thm:the linear evolution problem}Assume that $\mathsf{y}\in\mathscr{Y}$
and $g^{0}\in H^{6}(\mathbb{S})$. Then there exists a unique solution
$\mathsf{x}\in\mathscr{X}$ to (\ref{eq:aux linear problem}), which
satisfies the following estimate: 
\begin{equation}
\Vert\mathsf{x}\Vert_{\mathscr{X}}\leq C(\Vert g^{0}\Vert_{H^{6}(\mathbb{S})}+\Vert\mathsf{y}\Vert_{\mathscr{Y}}),\label{eq:est linear problem}
\end{equation}
for some constant $C$.\end{thm}
\begin{proof}
It will be given in Section \ref{sec:The linear evolution problem}.
\end{proof}
By Theorem \ref{thm:the linear evolution problem}, the operator $\mathcal{L}$
is invertible, so that we can define the operator $\mathcal{A}$:
\begin{equation}
\mathcal{A}[\mathsf{x}]:=\mathcal{L}^{-1}[\mathcal{Y}[\mathsf{x}]],\quad\mathcal{A}\colon\mathscr{X}\rightarrow\mathscr{X}.\label{s411}
\end{equation}
Our last main theorem verifies that the perturbed solution (\ref{eq:self-similar variables solution})
in the $\xi$-space converges to the ellipsoidal solution that is
deduced from Theorem \ref{thm:existence of ellipsoids}:
\begin{thm}[The nonlinear evolution problem]
\label{thm:the nonlinear evolution problem}If $g^{0}\in H^{6}(\mathbb{S})$,
then for a sufficiently small $\epsilon>0$, the nonlinear problem
\begin{equation}
\mathsf{x}=\mathcal{A}[\mathsf{x}],\label{eq:nonlinear problem}
\end{equation}
where $\mathcal{A}$ is defined by (\ref{s411}), has a unique solution
$\mathsf{x}\in\mathscr{X}$.\end{thm}
\begin{proof}
The proof is based on the contraction mapping principle. One can easily
check that 
\begin{equation}
\Vert\mathcal{Y}[\mathsf{x}]\Vert_{\mathscr{Y}}\leq C_{0}\epsilon\Vert\mathsf{x}\Vert_{\mathscr{X}}\label{s412}
\end{equation}
for any $\mathsf{x}\in\mathscr{X}$ such that $\Vert\mathsf{x}\Vert_{\mathscr{X}}\leq2C\Vert g^{0}\Vert_{H^{6}(\mathbb{S})}$,
where $C$ is the constant in (\ref{eq:est linear problem}) and $C_{0}$
is some appropriate constant. Therefore, using (\ref{eq:est linear problem}),
we deduce 
\[
\Vert\mathcal{A}[\mathsf{x}]\Vert_{\mathscr{X}}\leq C(\Vert g^{0}\Vert_{H^{6}(\mathbb{S})}+C_{0}\epsilon\Vert\mathsf{x}\Vert_{\mathscr{X}})
\]
so that $\mathcal{A}$ maps the ball of radius $R=2C\Vert g^{0}\Vert_{H^{6}(\mathbb{S})}$
into itself for a sufficiently small $\epsilon>0$. We now verify
that $\mathcal{A}$ is a strict contraction. By Theorem \ref{thm:the linear evolution problem},
for any $\tilde{\mathsf{x}}_{1}$ and $\tilde{\mathsf{x}}_{2}$ in
$\mathscr{X}$, there exist unique solutions $\mathsf{x}_{1}$ and
$\mathsf{x}_{2}$ in $\mathscr{X}$ such that $\mathcal{L}[\mathsf{x}_{1}]=\mathcal{Y}[\tilde{\mathsf{x}}_{1}]$
and $\mathcal{L}[\mathsf{x}_{2}]=\mathcal{Y}[\tilde{\mathsf{x}}_{2}]$,
respectively. As a result, for $\tilde{\mathsf{x}}_{1}$ and $\tilde{\mathsf{x}}_{2}$
in the ball of radius $R$, we conclude 
\begin{equation}
\Vert\mathsf{x}_{1}-\mathsf{x}_{2}\Vert_{\mathscr{X}}\leq C\Vert\mathcal{Y}[\tilde{\mathsf{x}}_{1}]-\mathcal{Y}[\tilde{\mathsf{x}}_{2}]\Vert_{\mathscr{Y}}\leq C_{0}\epsilon\Vert\tilde{\mathsf{x}}_{1}-\tilde{\mathsf{x}}_{2}\Vert_{\mathscr{X}}\label{s413}
\end{equation}
and hence $\mathcal{A}$ is a strict contraction for a sufficiently
small $\epsilon>0$. Finally, by the contraction mapping principle,
(\ref{eq:nonlinear problem}) has a unique solution in the Banach
space $\mathscr{X}$.

In fact, the proof of the inequalities (\ref{s412}) and (\ref{s413})
is quite lengthy. However, one can show the inequalities in exactly
the same way as Section 4 of \cite{Bazaliy:2003ic}, whose proof requires
the following multiplicative property of the norms of the spaces $\mathscr{X}$
and $\mathscr{Y}$: 
\[
\left\Vert fg\right\Vert \leq C_{1}\left\Vert f\right\Vert \left\Vert g\right\Vert 
\]
for some appropriate constant $C_{1}$ (see \cite{Friedman:2001vk}).
\end{proof}

\section{\label{sec:the ODE system}The system of ordinary differential equations}

In this section we solve the following auxiliary problem: 
\begin{align}
 & \Delta P=F^{1} &  & \textnormal{in}\,\mathbb{B},\label{eq:auxode1}\\
 & -\nabla P+\Delta\mathbf{U}=\mathbf{F}^{2} &  & \textnormal{in}\,\mathbb{B},\label{eq:auxode2}\\
 & \operatorname{div}\mathbf{U}=F^{3} &  & \textnormal{on}\,\mathbb{S},\label{eq:auxode3}\\
 & [-P\mathbf{I}+\nabla\mathbf{U}+(\nabla\mathbf{U})^{T}]\mathbf{e}_{r}-\sigma(g+\frac{1}{2}\Delta_{\omega}g)\mathbf{e}_{r}=\mathbf{F}^{4} &  & \textnormal{on}\,\mathbb{S},\label{eq:auxode4}\\
 & \int_{\mathbb{B}}\mathbf{U}dV=\mathbf{F}^{5},\label{eq:auxode5}\\
 & \int_{\mathbb{B}}\left(\mathbf{U}\times(r\sin\theta\cos\varphi,r\sin\theta\sin\varphi,r\cos\theta)\right)dV=\mathbf{F}^{6},\label{eq:auxode6}\\
 & \Delta\Phi=F^{7} &  & \textnormal{in}\,\operatorname{ext}\mathbb{B},\label{eq:auxode7}\\
 & \Phi-g=F^{8} &  & \textnormal{on}\,\mathbb{S},\label{eq:auxode8}\\
 & \Phi\rightarrow0 &  & \textnormal{as}\, r\rightarrow\infty.\label{eq:auxode9}
\end{align}
The spherical harmonics form a complete set in a suitably weighted
$L^{2}$ space. Likewise, the vector spherical harmonics form an orthonormal
basis of a suitably weighted $L^{2}$ space. Thus we write the solution
to (\ref{eq:auxode1})-(\ref{eq:auxode9}) as the following expansions
in terms of the spherical harmonics and the vector spherical harmonics:
\begin{equation}
\begin{cases}
\mathbf{U}=\sum_{l,m}\left(U_{l,m}^{V}\mathbf{V}_{l,m}+U_{l,m}^{X}\mathbf{X}_{l,m}+U_{l,m}^{W}\mathbf{W}_{l,m}\right),\\
P=\sum_{l,m}P_{l,m}Y_{l,m},\\
\Phi=\sum_{l,m}\Phi_{l,m}Y_{l,m},\\
g=\sum_{l,m}g_{l,m}Y_{l,m}.
\end{cases}\label{s415}
\end{equation}
The exact formulas and fundamental properties of these eigenfunctions
are given in Appendix \ref{sec:S Harmonics and Vector S Harmonics}.
Analogously, the non-homogeneous terms in (\ref{eq:auxode1})-(\ref{eq:auxode8})
have the following expansions:
\begin{equation}
\begin{cases}
F^{1}=\sum_{l,m}F_{l,m}^{1}Y_{l,m},\\
\mathbf{F}^{2}=\sum_{l,m}\left(F_{l,m}^{2;V}\mathbf{V}_{l,m}+F_{l,m}^{2;X}\mathbf{X}_{l,m}+F_{l,m}^{2;W}\mathbf{W}_{l,m}\right)\\
F^{3}=\sum_{l,m}F_{l,m}^{3}Y_{l,m},\\
\mathbf{F}^{4}=\sum_{l,m}\left(F_{l,m}^{4;V}\mathbf{V}_{l,m}+F_{l,m}^{4;X}\mathbf{X}_{l,m}+F_{l,m}^{4;W}\mathbf{W}_{l,m}\right)\\
F^{7}=\sum_{l,m}F_{l,m}^{7}Y_{l,m},\\
F^{8}=\sum_{l,m}F_{l,m}^{8}Y_{l,m}.
\end{cases}\label{s416}
\end{equation}
Moreover, we write 
\begin{equation}
\begin{cases}
\mathbf{F}^{5}=F_{1}^{5}\mathbf{e}_{1}+F_{2}^{5}\mathbf{e}_{2}+F_{3}^{5}\mathbf{e}_{3},\\
\mathbf{F}^{6}=F_{1}^{6}\mathbf{e}_{1}+F_{2}^{6}\mathbf{e}_{2}+F_{3}^{6}\mathbf{e}_{3},
\end{cases}\label{s417}
\end{equation}
where $\{\mathbf{e}_{1},\mathbf{e}_{2},\mathbf{e}_{3}\}$ is the standard
basis for $\mathbb{R}^{3}$. We construct the explicit forms $U_{l,m}^{V}$,
$U_{l,m}^{X}$, $,U_{l,m}^{W}$, $P_{l,m}$, $\Phi_{l,m}$ in terms
of $g_{l,m}$ and the coefficients of the expansions in (\ref{s416}),
(\ref{s417}). As a consequence, we deduce the following expansions
on $r=1$, which are in (\ref{eq:dce boundary condition}), (\ref{eq:dct kinematic condition}),
and (\ref{eq:dct re kinematic}): 
\[
\mathbf{U}\cdot\mathbf{e}_{r}=\sum_{l,m}(\mathbf{U}\cdot\mathbf{e}_{r})_{l,m}Y_{l,m},
\]
\[
\frac{\partial\Phi}{\partial r}=\sum_{l,m}\left(\frac{\partial\Phi}{\partial r}\right)_{l,m}Y_{l,m}.
\]

\subsection{The derivation of $\left(\mathbf{U}\cdot\mathbf{e}_{r}\right)_{l,m}$
on $r=1$}

By (\ref{eq:laplacian FYlm}), the equation (\ref{eq:auxode1}) becomes
\begin{equation}
L_{l}P_{l,m}=F_{l,m}^{1},\quad0<r<1.\label{eq:laplacian P ode lm}
\end{equation}
Furthermore, by (\ref{eq:laplacian FYlm})-(\ref{eq:gradient FYlm}),
the equation (\ref{eq:auxode2}) becomes 
\begin{equation}
\begin{cases}
L_{l+1}U_{l,m}^{V}-\left(\frac{l+1}{2l+1}\right)^{1/2}\left(-\frac{\partial}{\partial r}+\frac{l}{r}\right)P_{l,m}=F_{l,m}^{2;V}, & r<1,\\
L_{l}U_{l,m}^{X}=F_{l,m}^{2;X}, & r<1,\\
L_{l-1}U_{l,m}^{W}-\left(\frac{l}{2l+1}\right)^{1/2}\left(\frac{\partial}{\partial r}+\frac{l+1}{r}\right)P_{l,m}=F_{l,m}^{2;W}, & r<1.
\end{cases}\label{eq:stokes ode lm}
\end{equation}
The general solution to (\ref{eq:laplacian P ode lm}) and (\ref{eq:stokes ode lm})
is 
\begin{equation}
\begin{cases}
P_{l,m}(r)=P_{l,m}^{1}r^{l}+H_{l,m}^{P},\\
U_{l,m}^{V}(r)=V_{l,m}^{1}r^{l+1}+H_{l,m,}^{V}\\
U_{l,m}^{X}(r)=X_{l,m}^{1}r^{l}+H_{l,m}^{X},\\
U_{l,m}^{W}(r)=W_{l,m}^{1}r^{l-1}+\frac{1}{2}\left(\frac{l}{2l+1}\right)^{1/2}P_{l,m}^{1}r^{l+1}+H_{l,m}^{W},
\end{cases}\label{eq:general solution of stokes ode lm}
\end{equation}
where the solution to the non-homogeneous part is as follows: 
\[
H_{l,m}^{P}=-r^{l}\int_{r}^{1}\frac{s^{-l+1}}{2l+1}F_{l,m}^{1}ds-r^{-l-1}\int_{0}^{r}\frac{s^{l+2}}{2l+1}F_{l,m}^{1}ds,
\]
\begin{align*}
H_{l,m}^{V} & =-r^{l+1}\int_{r}^{1}\frac{s^{-l}}{2l+3}F_{l,m}^{2;V}ds-r^{-l-2}\int_{0}^{r}\frac{s^{l+3}}{2l+3}F_{l,m}^{2;V}ds\\
 & +\left(\frac{l+1}{2l+1}\right)^{1/2}\frac{1}{2l+3}\left[\frac{1}{2l+1}(r^{-l}-r^{l+1})\int_{0}^{r}s^{l+2}F_{l,m}^{1}ds\right.\\
 & \left.+\frac{1}{2l+1}r^{l+1}\int_{r}^{1}(s^{-l+1}-s^{l+2})F_{l,m}^{1}ds+r^{-l-2}\int_{0}^{r}s^{l+2}\frac{r^{2}-s^{2}}{2}F_{l,m}^{1}ds\right],
\end{align*}
\[
H_{l,m}^{X}=-r^{l}\int_{r}^{1}\frac{s^{-l+1}}{2l+1}F_{l,m}^{2;X}ds-r^{-l-1}\int_{0}^{r}\frac{s^{l+2}}{2l+1}F_{l,m}^{2;X}ds,
\]
\begin{align*}
H_{l,m}^{W} & =-r^{l-1}\int_{r}^{1}\frac{s^{-l+2}}{2l-1}F_{l,m}^{2;W}ds-r^{-l}\int_{0}^{r}\frac{s^{l+1}}{2l-1}F_{l,m}^{2;W}ds\\
 & +\left(\frac{l}{2l+1}\right)^{1/2}\frac{1}{2l-1}\left[r^{l-1}\int_{r}^{1}s^{-l+1}\frac{s^{2}-r^{2}}{2}F_{l,m}^{1}ds\right.\\
 & \left.+\frac{r^{-l}}{2l+1}\left(\int_{0}^{r}s^{l+2}F_{l,m}^{1}ds+\int_{r}^{1}s^{-l+1}r^{2l+1}F_{l,m}^{1}ds\right)\right].
\end{align*}

We deduce (\ref{eq:auxode3})-(\ref{eq:auxode6}) in terms of the
spherical harmonics expansions. By (\ref{eq:divergence FVlm})-(\ref{eq:divergence FWlm}),
the boundary condition (\ref{eq:auxode3}) reduces to 
\begin{equation}
-\left(\frac{l+1}{2l+1}\right)^{1/2}\left(\frac{\partial}{\partial r}+l+2\right)U_{l,m}^{V}+\left(\frac{l}{2l+1}\right)^{1/2}\left(\frac{\partial}{\partial r}-l+1\right)U_{l,m}^{W}=F_{l,m}^{3},\quad r=1.\label{eq:div lm}
\end{equation}
A direct computation from Lemma 4.1 in \cite{Friedman:2002en} shows
that (\ref{eq:auxode4}) becomes 
\begin{align}
 & \left(\frac{3l+2}{2l+1}\frac{\partial}{\partial r}-\frac{l^{2}+2l}{2l+1}\right)U_{l,m}^{V}-\frac{l^{1/2}(l+1)^{1/2}}{2l+1}\left(\frac{\partial}{\partial r}-l+1\right)U_{l,m}^{W}+\left(\frac{l+1}{2l+1}\right)^{1/2}P_{l,m}\notag\\
 & +\sigma\left(\frac{l+1}{2l+1}\right)^{1/2}\left(1-\frac{l^{2}+l}{2}\right)g_{l,m}=F_{l,m}^{4;V},\quad r=1,\label{eq:normal stress Vlm}
\end{align}
\begin{equation}
\left(\frac{\partial}{\partial r}-1\right)U_{l,m}^{X}=F_{l,m}^{4;X},\quad r=1,\label{eq:normal stress Xlm}
\end{equation}
\begin{align}
 & -\frac{l^{1/2}(l+1)^{1/2}}{2l+1}\left(\frac{\partial}{\partial r}+l+2\right)U_{l,m}^{V}+\left(\frac{3l+1}{2l+1}\frac{\partial}{\partial r}+\frac{l^{2}-1}{2l+1}\right)U_{l,m}^{W}-\left(\frac{l}{2l+1}\right)^{1/2}P_{l,m}\notag\\
 & -\sigma\left(\frac{l}{2l+1}\right)^{1/2}\left(1-\frac{l^{2}+l}{2}\right)g_{l,m}=F_{l,m}^{4;W},\quad r=1.\label{eq:normal stress Wlm}
\end{align}
Finally, by using Lemma 8.1 in \cite{Friedman:2002en} with its proof,
the constraints (\ref{eq:auxode5}) and (\ref{eq:auxode6}) reduce
to 
\begin{equation}
\sum_{m,j}\frac{1}{\sqrt{3}}\left(U_{1,m}^{W}(1)-\int_{0}^{1}r^{3}\frac{\partial U_{1,m}^{W}}{\partial r}dr\right)\left(\int_{\mathbb{S}}x_{j}Y_{1,m}dS\right)\mathbf{e}_{j}=\sum_{j}F_{j}^{5}\mathbf{e}_{j},\label{eq:mmt lm}
\end{equation}
\begin{equation}
\sum_{m,j}-i\sqrt{2}\left(\int_{0}^{1}r^{3}U_{1,m}^{X}dr\right)\left(\int_{\mathbb{S}}x_{j}Y_{1,m}dS\right)\mathbf{e}_{j}=\sum_{j}F_{j}^{6}\mathbf{e}_{j}.\label{eq:ang.mmt lm}
\end{equation}

Substituting (\ref{eq:general solution of stokes ode lm}) into the
boundary conditions (\ref{eq:div lm})-(\ref{eq:normal stress Wlm}),
we are led to the following linear system of equations for $P_{l,m}^{1}$,
$V_{l,m}^{1}$, $X_{l,m}^{1}$, $W_{l,m}^{1}$ and $g_{l,m}$: 
\begin{equation}
-\left(\frac{l+1}{2l+1}\right)^{1/2}(2l+3)V_{l,m}^{1}+\left(\frac{l}{2l+1}\right)P_{l,m}^{1}=J_{l,m}^{div},\label{eq:linear sys div}
\end{equation}
\begin{equation}
\frac{2l^{2}+3l+2}{2l+1}V_{l,m}^{1}+\frac{l+1}{2l+1}\left(\frac{l+1}{2l+1}\right)^{1/2}P_{l,m}^{1}+\sigma\left(\frac{l+1}{2l+1}\right)^{1/2}\left(1-\frac{l^{2}+l}{2}\right)g_{l,m}=J_{l,m}^{nsb;V},\label{eq:linear sys nsbV}
\end{equation}
\begin{equation}
(l-1)X_{l,m}^{1}=J_{l,m}^{nsb;X},\label{eq:linear sys nsbX}
\end{equation}
\begin{equation}
-\frac{l^{1/2}(l+1)^{1/2}}{2l+1}(2l+3)V_{l,m}^{1}+2(l-1)W_{l,m}^{1}+\frac{2l^{2}-1}{2l+1}\left(\frac{l}{2l+1}\right)^{1/2}P_{l,m}^{1}-\sigma\left(\frac{l}{2l+1}\right)^{1/2}\left(1-\frac{l^{2}+l}{2}\right)g_{l,m}=J_{l,m}^{nsb;W}.\label{eq:linear sys nsbW}
\end{equation}
The right-hand sides of (\ref{eq:linear sys div})-(\ref{eq:linear sys nsbW})
consist of $H_{l,m}^{P}$, $H_{l,m}^{V}$, $H_{l,m}^{X}$, $H_{l,m}^{W}$,
$F_{l,m}^{3}$, $F_{l,m}^{4;V}$, $F_{l,m}^{4;X}$ and $F_{l,m}^{4;W}$.
In the case $l=1$ we will use (\ref{eq:mmt lm}) and (\ref{eq:ang.mmt lm})
instead of (\ref{eq:linear sys nsbX}) and (\ref{eq:linear sys nsbW}),
because $X_{1,m}^{1}$ and $W_{1,m}^{1}$ do not appear in (\ref{eq:normal stress Xlm})
and (\ref{eq:normal stress Wlm}). In the case $l=0$ the relation
$\mathbf{X}_{0,0}=\mathbf{W}_{0,0}=0$ holds so that $U_{0,0}^{X}=U_{0,0}^{W}=0$.
Consider first the case $l\geq2$. Solving (\ref{eq:linear sys div})-(\ref{eq:linear sys nsbW}),
we conclude 
\begin{equation}
P_{l,m}^{1}=\sigma\frac{2l^{4}+7l^{3}+4l^{2}-7l-6}{4l^{2}+8l+6}g_{l,m}+\frac{2l^{2}+3l+2}{2l^{2}+4l+3}J_{l,m}^{div}+\left(\frac{l+1}{2l+1}\right)^{1/2}\frac{4l^{2}+8l+3}{2l^{2}+4l+3}J_{l,m}^{nsb;V},\label{eq:P1 lm}
\end{equation}
\begin{equation}
V_{l,m}^{1}=\sigma\left(\frac{l+1}{2l+1}\right)^{1/2}\frac{l^{3}+l^{2}-2l}{4l^{2}+8l+6}g_{l,m}-\left(\frac{l+1}{2l+1}\right)^{1/2}\frac{l+1}{2l^{2}+4l+3}J_{l,m}^{div}+\frac{l}{2l^{2}+4l+3}J_{l,m}^{nsb;V},\label{eq:V1 lm}
\end{equation}
\begin{equation}
X_{l,m}^{1}=\frac{1}{l-1}J_{l,m}^{nsb;X},\label{eq:X1 lm}
\end{equation}
\begin{align}
W_{l,m}^{1} & =-\sigma\left(\frac{l}{2l+1}\right)^{1/2}\frac{2l^{4}+9l^{3}+12l^{2}+4l}{8l^{2}+16l+12}g_{l,m}\label{eq:W1 lm}\\
 & -\left(\frac{l}{2l+1}\right)^{1/2}\frac{4l^{3}+6l^{2}+6l+2}{8l^{3}+8l^{2}-4l-12}J_{l,m}^{div}-\frac{l^{1/2}(l+1)^{1/2}(2l+3)}{4l^{2}+8l+6}J_{l,m}^{nsb;V}+\frac{1}{2l-2}J_{l,m}^{nsb;W}.\notag
\end{align}
In the case $l=1$, from (\ref{eq:linear sys div}) and (\ref{eq:linear sys nsbV})
we have 
\begin{equation}
P_{1,m}^{1}=\frac{7}{9}J_{1,m}^{div}+\frac{5}{3}\sqrt{\frac{2}{3}}J_{1,m}^{nsb;V},\label{eq:P1 1m}
\end{equation}
\begin{equation}
V_{1,m}^{1}=-\frac{2}{9}\sqrt{\frac{2}{3}}J_{1,m}^{div}+\frac{1}{9}J_{1,m}^{nsb;V}.\label{eq:V1 1m}
\end{equation}
Substituting $U_{1,m}^{X}$ and $U_{1,m}^{W}$ in (\ref{eq:general solution of stokes ode lm})
into the constraints (\ref{eq:mmt lm}) and (\ref{eq:ang.mmt lm})
we deduce the following linear system of equations for $X_{1,m}^{1}$
and $W_{1,m}^{1}$: 
\begin{equation}
\sqrt{\frac{2\pi}{3}}\left[\frac{1}{\sqrt{3}}W_{1,-1}^{1}+\frac{1}{10}P_{1,-1}^{1}-\frac{1}{\sqrt{3}}W_{1,1}^{1}-\frac{1}{10}P_{1,1}^{1}\right]=J_{1}^{mmt},\label{eq:linear sys mmt1}
\end{equation}
\begin{equation}
-i\sqrt{\frac{2\pi}{3}}\left[\frac{1}{\sqrt{3}}W_{1,-1}^{1}+\frac{1}{10}P_{1,-1}^{1}+\frac{1}{\sqrt{3}}W_{1,1}^{1}+\frac{1}{10}P_{1,1}^{1}\right]=J_{2}^{mmt},\label{eq:linear sys mmt2}
\end{equation}
\begin{equation}
2\sqrt{\frac{\pi}{3}}\left[\frac{1}{\sqrt{3}}W_{1,0}^{1}+\frac{1}{10}P_{1,0}^{1}\right]=J_{3}^{mmt},\label{eq:linear sys mmt3}
\end{equation}
\begin{equation}
-i\frac{2}{5}\sqrt{\frac{\pi}{3}}\left(X_{1,-1}^{1}-X_{1,1}^{1}\right)=J_{1}^{ang.mmt},\label{eq:linear sys ang.mmt1}
\end{equation}
\begin{equation}
-\frac{2}{5}\sqrt{\frac{\pi}{3}}\left(X_{1,-1}^{1}+X_{1,1}^{1}\right)=J_{2}^{ang.mmt},\label{eq:linear sys ang.mmt2}
\end{equation}
\begin{equation}
-i\frac{2}{5}\sqrt{\frac{2\pi}{3}}X_{1,0}^{1}=J_{3}^{ang.mmt},\label{eq:linear sys ang.mmt3}
\end{equation}
where $\left\{ P_{1,m}^{1}\right\} _{|m|\leq1}$ are given by (\ref{eq:P1 1m})
and the right-hand sides of (\ref{eq:linear sys mmt1})-(\ref{eq:linear sys ang.mmt3})
consist of $H_{1,m}^{X}$, $H_{1,m}^{W}$, $F_{j}^{5}$, and $F_{j}^{6}$.
From (\ref{eq:linear sys mmt1})-(\ref{eq:linear sys ang.mmt3}) we
obtain 
\begin{equation}
X_{1,-1}^{1}=i\frac{5}{4}\sqrt{\frac{3}{\pi}}J_{1}^{ang.mmt}-\frac{5}{4}\sqrt{\frac{3}{\pi}}J_{2}^{ang.mmt},\label{eq:X1 -11}
\end{equation}
\begin{equation}
X_{1,0}^{1}=i\frac{5}{2}\sqrt{\frac{3}{2\pi}}J_{3}^{ang.mmt},\label{eq:X1 10}
\end{equation}
\begin{equation}
X_{1,1}^{1}=-i\frac{5}{4}\sqrt{\frac{3}{\pi}}J_{1}^{ang.mmt}-\frac{5}{4}\sqrt{\frac{3}{\pi}}J_{2}^{ang.mmt},\label{eq:X1 11}
\end{equation}
\begin{equation}
W_{1,-1}^{1}=-\frac{7}{30\sqrt{3}}J_{1,-1}^{div}-\frac{1}{3\sqrt{2}}J_{1,-1}^{nsb;V}+\frac{3}{2\sqrt{2\pi}}J_{1}^{mmt}+i\frac{3}{2\sqrt{2\pi}}J_{2}^{mmt},\label{eq:W1 -11}
\end{equation}
\begin{equation}
W_{1,0}^{1}=-\frac{7}{30\sqrt{3}}J_{1,0}^{div}-\frac{1}{3\sqrt{2}}J_{1,0}^{nsb;V}+\frac{3}{2\sqrt{\pi}}J_{3}^{mmt},\label{eq:W1 10}
\end{equation}
\begin{equation}
W_{1,1}^{1}=-\frac{7}{30\sqrt{3}}J_{1,1}^{div}-\frac{1}{3\sqrt{2}}J_{1,1}^{nsb;V}-\frac{3}{2\sqrt{2\pi}}J_{1}^{mmt}+i\frac{3}{2\sqrt{2\pi}}J_{2}^{mmt}.\label{eq:W1 11}
\end{equation}
In the case $l=0$ we already know that 
\begin{equation}
X_{0,0}^{1}=W_{0,0}^{1}=0,\label{eq:X1 00 W1 00}
\end{equation}
because $U_{0,0}^{X}=U_{0,0}^{W}=0$. Moreover, solving (\ref{eq:linear sys div})
and (\ref{eq:linear sys nsbV}) we have 
\begin{equation}
P_{0,0}^{1}=-\sigma g_{0,0}+\frac{2}{3}J_{0,0}^{div}+J_{0,0}^{nsb;V},\label{eq:P1 00}
\end{equation}
\begin{equation}
V_{0,0}^{1}=-\frac{1}{3}J_{0,0}^{div}.\label{eq:V1 00}
\end{equation}

By (\ref{eq:def Vlm})-(\ref{eq:def Wlm}), it is easily seen that
\begin{equation}
\left(\mathbf{U}\cdot\mathbf{e}_{r}\right)_{l,m}=-\left(\frac{l+1}{2l+1}\right)^{1/2}U_{l,m}^{V}+\left(\frac{l}{2l+1}\right)^{1/2}U_{l,m}^{W}.\label{eq:Uer lm primary}
\end{equation}
Substituting $U_{l,m}^{V}$ and $U_{l,m}^{W}$ in (\ref{eq:general solution of stokes ode lm})
into (\ref{eq:Uer lm primary}) on $r=1$ we arrive at 
\begin{equation}
\left(\mathbf{U}\cdot\mathbf{e}_{r}\right)_{l,m}=-\left(\frac{l+1}{2l+1}\right)^{1/2}V_{l,m}^{1}+\left(\frac{l}{2l+1}\right)^{1/2}\left[W_{l,m}^{1}+\frac{1}{2}\left(\frac{l}{2l+1}\right)^{1/2}P_{l,m}^{1}\right]+J_{l,m}^{\mathbf{U}\cdot\mathbf{e}_{r}},\label{eq:Uer lm r 1}
\end{equation}
where $J_{l,m}^{\mathbf{U}\cdot\mathbf{e}_{r}}$ consist of $H_{l,m}^{V}$
and $H_{l,m}^{W}$. Finally, substituting (\ref{eq:P1 lm}), (\ref{eq:V1 lm}),
(\ref{eq:W1 lm})-(\ref{eq:V1 1m}), (\ref{eq:W1 -11})-(\ref{eq:V1 00})
into (\ref{eq:Uer lm r 1}) we conclude the following lemma:
\begin{lem}
\label{lem:Uer r 1}The vector $\mathbf{U}=\sum_{l,m}\left(U_{l,m}^{V}\mathbf{V}_{l,m}+U_{l,m}^{X}\mathbf{X}_{l,m}+U_{l,m}^{W}\mathbf{W}_{l,m}\right)$
in the solution of (\ref{eq:auxode1})-(\ref{eq:auxode6}), evaluated
at $r=1$, satisfies the following relations: 
\begin{align*}
\left(\mathbf{U}\cdot\mathbf{e}_{r}\right)_{l,m} & =-\sigma l\frac{2l^{2}+5l+2}{8l^{2}+16l+12}g_{l,m}\\
 & -\frac{l+2}{4l^{3}+4l^{2}-2l-6}J_{l,m}^{div}-\left(\frac{l+1}{2l+1}\right)^{1/2}\frac{l}{2l^{2}+4l+3}J_{l,m}^{nsb;V}+\left(\frac{l}{2l+1}\right)^{1/2}\frac{1}{2l-2}J_{l,m}^{nsb;W}+J_{l,m}^{\mathbf{U}\cdot\mathbf{e}_{r}},\quad l\geq2,
\end{align*}
\[
\left(\mathbf{U}\cdot\mathbf{e}_{r}\right)_{1,-1}=\frac{1}{5}J_{1,-1}^{div}+\frac{1}{2}\sqrt{\frac{3}{2\pi}}J_{1}^{mmt}+i\frac{1}{2}\sqrt{\frac{3}{2\pi}}J_{2}^{mmt}+J_{1,-1}^{\mathbf{U}\cdot\mathbf{e}_{r}},
\]
\[
\left(\mathbf{U}\cdot\mathbf{e}_{r}\right)_{1,0}=\frac{1}{5}J_{1,0}^{div}+\frac{1}{2}\sqrt{\frac{3}{\pi}}J_{3}^{mmt}+J_{1,0}^{\mathbf{U}\cdot\mathbf{e}_{r}},
\]
\[
\left(\mathbf{U}\cdot\mathbf{e}_{r}\right)_{1,1}=\frac{1}{5}J_{1,1}^{div}-\frac{1}{2}\sqrt{\frac{3}{2\pi}}J_{1}^{mmt}+i\frac{1}{2}\sqrt{\frac{3}{2\pi}}J_{2}^{mmt}+J_{1,1}^{\mathbf{U}\cdot\mathbf{e}_{r}},
\]
\[
\left(\mathbf{U}\cdot\mathbf{e}_{r}\right)_{0,0}=\frac{1}{3}J_{0,0}^{div}+J_{0,0}^{\mathbf{U}\cdot\mathbf{e}_{r}}.
\]

\end{lem}

\subsection{The derivation of $\left(\frac{\partial\Phi}{\partial r}\right)_{l,m}$
on $r=1$}

By (\ref{eq:laplacian FYlm}), the equation (\ref{eq:auxode7}) becomes
\begin{equation}
L_{l}\Phi_{l,m}=F_{l,m}^{7},\quad r>1.\label{eq:laplacian Philm}
\end{equation}
In addition, (\ref{eq:auxode8}) and (\ref{eq:auxode9}) reduce to
\begin{align}
 & \Phi_{l,m}-g_{l,m}=F_{l,m}^{8},\quad r=1,\label{eq:Phi lm r 1}\\
 & \Phi_{l,m}\rightarrow0,\quad\textnormal{as}\, r\rightarrow\infty.\label{eq:Phi lm r infty}
\end{align}
We compute directly the solution to (\ref{eq:laplacian Philm})-(\ref{eq:Phi lm r infty}),
and hence 
\begin{equation}
\Phi_{l,m}(r)=\frac{1}{r^{l+1}}g_{l,m}+\frac{1}{r^{l+1}}F_{l,m}^{8}+H_{l,m}^{\Phi},\label{eq:vaporp sol lm}
\end{equation}
where 
\[
H_{l,m}^{\Phi}=-\frac{1}{r^{l+1}}\int_{1}^{r}\rho^{2l}\int_{\rho}^{\infty}\frac{1}{s^{l-1}}F_{l,m}^{7}dsd\rho.
\]
In conclusion one can easily show the following lemma:
\begin{lem}
\label{lem:Phir r 1}The solution $\Phi=\sum_{l,m}\Phi_{l,m}Y_{l,m}$
to (\ref{eq:auxode7})-(\ref{eq:auxode9}) satisfies the following
relation at $r=1$: 
\begin{equation}
\left(\frac{\partial\Phi}{\partial r}\right)_{l,m}=-(l+1)g_{l,m}-(l+1)F_{l,m}^{8}-\int_{1}^{\infty}\frac{1}{s^{l-1}}F_{l,m}^{7}ds.\label{eq:vaporflux lm}
\end{equation}

\end{lem}

\section{\label{sec:Existence of ellipsoids}Existence of ellipsoids}

In this section we prove Theorem \ref{thm:existence of ellipsoids}.
By Lemma \ref{lem:Phir r 1}, one can easily verify that (\ref{eq:dce boundary condition})
reduces to the following nonlinear system of equations on $r=1$ for
$\Phi_{l,m}^{\mathfrak{E}}$, $g_{l,m}^{\mathfrak{E}}$, and $t_{0}$:
\begin{equation}
(l-2)g_{l,m}^{\mathfrak{E}}=\left(f^{10}[\Phi^{\mathfrak{E}},g^{\mathfrak{E}},t_{0}]\right)_{l,m}+K_{l,m}^{\mathfrak{E}},\quad l\geq1,\label{eq:dce kinematic lm}
\end{equation}
\[
-2g_{0,0}^{\mathfrak{E}}+t_{0}=\left(f^{10}[\Phi^{\mathfrak{E}},g^{\mathfrak{E}},t_{0}]\right)_{0,0}+K_{0,0}^{\mathfrak{E}},
\]
where $K_{l,m}^{\mathfrak{E}}$ are given in terms of $\left(f^{7}[\Phi^{\mathfrak{E}},g^{\mathfrak{E}}]\right)_{l,m}$
and $\left(f^{8}[g^{\mathfrak{E}}]\right)_{l,m}$. Main emphasis is
on the case $l=2$ for (\ref{eq:dce kinematic lm}): 
\[
0=\left(f^{10}[\Phi^{\mathfrak{E}},g^{\mathfrak{E}},t_{0}]\right)_{2,m}+K_{2,m}^{\dce}.
\]
Indeed, it is by no means a trivial task to prove Theorem \ref{thm:existence of ellipsoids}
because of the complexity coming from the nonlinearity of $f^{7}[\Phi^{\mathfrak{E}},g^{\mathfrak{E}}]$,
$f^{8}[g^{\mathfrak{E}}]$, and $f^{10}[\Phi^{\mathfrak{E}},g^{\mathfrak{E}},t_{0}]$.
In order to resolve this drawback, we use a new approach based on
the particular structure of (\ref{eq:dce stokes system})-(\ref{eq:dce boundary condition}).

\subsection{Null quadrature domains}

We first introduce null quadrature domains and the theorem without
its proof on a characterization of null quadrature domains.
\begin{defn}
An open set $\Xi\subset\mathbb{R}^{n}$ is called a \textit{null quadrature
domain }if 
\[
\int_{\Xi}udV=0
\]
for all harmonic and integrable functions $u$ in $\Xi$.\end{defn}
\begin{thm}
\label{thm:Friedman-Sakai}\cite{Friedman:1986hn} Assume that $\Xi$
has a bounded complement. Then $\Xi$ is a null quadrature domain
if and only if the boundary of $\Xi$ is an ellipsoid.
\end{thm}
We are now ready to prove the following lemma:
\begin{lem}
\label{lem:if then ellipsoid}If the overdetermined problem (\ref{eq:dce stokes system})-(\ref{eq:dce boundary condition})
has a classical solution, then $\{r=1+\epsilon g^{\dce}\}$ is an
ellipsoid.\end{lem}
\begin{proof}
From (\ref{eq:dce boundary condition}), $g^{\dce}$ is determined
only by the solution to (\ref{eq:dce vapor}). We can verify the well-posedness
of (\ref{eq:dce stokes system}) in a similar fashion as in \cite{Friedman:2002io}
and \cite{Friedman:2002en} when $g^{\dce}$ is given. Therefore it
is sufficient to consider (\ref{eq:dce vapor}) and (\ref{eq:dce boundary condition}).
In the same way as in Section \ref{sec:Reformulation of the problem}
we easily deduce that (\ref{eq:dce vapor})-(\ref{eq:dce boundary condition})
is equivalent to the following system: 
\begin{equation}
\begin{cases}
\Delta\phi=0, & r>R^{+}(t)(1+\epsilon g^{\dce}),\\
\phi=0, & r=R^{+}(t)(1+\epsilon g^{\dce}),\\
\phi\rightarrow-1 & \textnormal{as}\, r\rightarrow\infty,
\end{cases}\label{eq:equiv ellipsoid vapor}
\end{equation}
\begin{equation}
v_{\mathbf{n}}=-\frac{1}{2}\frac{\partial\phi}{\partial\mathbf{n}},\quad r=R^{+}(t)(1+\epsilon g^{\dce}),\label{eq:equiv ellipsoid kinematic}
\end{equation}
\begin{equation}
\phi=R^{+}(t)/r-1+\epsilon\Phi^{\dce},\label{eq:equiv ellipsoid solution form}
\end{equation}
where $R^{+}(t)=[1+t/(1+\epsilon t_{0}Y_{0,0})]^{1/2}$. Let $u$
be any harmonic and integrable function in $\Xi(t):=\{r>R^{+}(t)(1+\epsilon g^{\dce})\}$.
It is well known that if $f$ is a harmonic function in a neighborhood
of infinity in $\mathbb{R}^{3}$, bounded by $O(1/r)$, then there
exists a constant $c$ such that
\[
f=\frac{c}{r}+O\left(\frac{1}{r^{2}}\right),\quad\nabla f=\nabla\left(\frac{c}{r}\right)+O\left(\frac{1}{r^{3}}\right)\quad\textnormal{as}\quad r\rightarrow\infty
\]
(see \cite{DiBenedetto:1986jt}). By using this property and Green's
formula, if (\ref{eq:equiv ellipsoid vapor})-(\ref{eq:equiv ellipsoid solution form})
has a classical solution, we have 
\begin{align*}
\frac{d}{dt}\int_{\Xi(t)}udV & =\int_{r=R^{+}(t)(1+\epsilon g^{\dce})}uv_{\mathbf{n}}dS=-\frac{1}{2}\int_{r=R^{+}(t)(1+\epsilon g^{\dce})}u\frac{\partial\phi}{\partial\mathbf{n}}dS\\
 & =\frac{1}{2}\int_{r=M}\left(u\frac{\partial\phi}{\partial\mathbf{n}}-\phi\frac{\partial u}{\partial\mathbf{n}}\right)dV\rightarrow0\quad\textnormal{as}\quad M\rightarrow\infty
\end{align*}
(as in \cite{Richardson:1981he}, \cite{DiBenedetto:1986jt}, \cite{Friedman:1986hn},
\cite{Howison:1986ez}). It immediately follows that for $0<t<T$,
\[
\int_{\Xi(t)}udV=\int_{\Xi(T)}udV\rightarrow0\quad\textnormal{as}\quad T\rightarrow\infty,
\]
and hence $\Xi(t)$ is a null quadrature domain. Therefore, by Theorem
\ref{thm:Friedman-Sakai}, $\{r=1+\epsilon g^{\dce}\}$ is an ellipsoid.
\end{proof}

\subsection{The bijective representation of ellipsoids near the unit sphere}

We show the following lemma which will be used to prove Theorem \ref{thm:existence of ellipsoids}:
\begin{lem}
\label{lem:bijective lem}For a sufficiently small $\epsilon>0$,
there exists a unique ellipsoid $\{r=1+\epsilon g^{\dce}\}$ which
satisfies the assumption on $g^{\dce}$ in Theorem \ref{thm:existence of ellipsoids}.\end{lem}
\begin{proof}
An arbitrarily oriented ellipsoid near the unit sphere, centered at
the origin, is defined by the equation 
\begin{equation}
\begin{pmatrix}x\\
y\\
z
\end{pmatrix}\begin{pmatrix}1+\epsilon\alpha_{11} & \epsilon\alpha_{12} & \epsilon\alpha_{13}\\
\epsilon\alpha_{12} & 1+\epsilon\alpha_{22} & \epsilon\alpha_{23}\\
\epsilon\alpha_{13} & \epsilon\alpha_{23} & 1+\epsilon\alpha_{33}
\end{pmatrix}\begin{pmatrix}x & y & z\end{pmatrix}=1,\label{eq:s521}
\end{equation}
where $\alpha_{ij}\in\mathbb{R}$ and the 3-by-3 matrix is positive
definite. In the spherical coordinate system, (\ref{eq:s521}) becomes
\[
r(\theta,\varphi)=1-\epsilon F[\{\alpha_{ij}\}]/2+G[\{\alpha_{ij}\}],
\]
where
\[
F[\{\alpha_{ij}\}]=\alpha_{33}\cos^{2}\theta+\alpha_{11}\cos^{2}\varphi\sin^{2}\theta+\alpha_{13}\cos\varphi\sin2\theta+\alpha_{23}\sin2\theta\sin\varphi+\alpha_{22}\sin^{2}\theta\sin^{2}\varphi+\alpha_{12}\sin^{2}\theta\sin2\varphi
\]
and 
\[
G[\{\alpha_{ij}\}]=1/(1+\epsilon F[\{\alpha_{ij}\}])^{1/2}-1+\epsilon F[\{\alpha_{ij}\}]/2.
\]
By using (\ref{eq:A1}), we easily deduce that $F[\{\alpha_{ij}\}]$
is the linear combination of $Y_{0,0}$ and $Y_{2,m}$ so that 
\begin{align}
r(\theta,\varphi) & =1+\epsilon\frac{\sqrt{\pi}}{30}[-10(\alpha_{11}+\alpha_{22}+\alpha_{33})Y_{0,0}+\sqrt{30}(-\alpha_{11}+\alpha_{22}-i2\alpha_{12})Y_{2,-2}-2\sqrt{30}(\alpha_{13}+i\alpha_{23})Y_{2,-1}\label{eq:ellipsoid r function Ylm}\\
 & +2\sqrt{5}(\alpha_{11}+\alpha_{22}-2\alpha_{33})Y_{2,0}+2\sqrt{30}(\alpha_{13}-i\alpha_{23})Y_{2,1}+\sqrt{30}(-\alpha_{11}+\alpha_{22}+i2\alpha_{12})Y_{2,2}]+G[\{\alpha_{ij}\}].\nonumber 
\end{align}
The coefficients of $Y_{0,0}$ and $Y_{2,m}$ in (\ref{eq:ellipsoid r function Ylm})
correspond to $1+\epsilon g_{0,0}^{\dce}$ and $\left\{ \epsilon g_{2,m}^{\dce}\right\} _{|m|\leq2}$,
respectively. Hence, from (\ref{eq:ellipsoid r function Ylm}) we
define the following system: 
\begin{equation}
\mathbf{f}\left(\{\alpha_{ij}\},\epsilon\right):=A\begin{pmatrix}\alpha_{11}\\
\alpha_{22}\\
\alpha_{33}\\
\alpha_{12}\\
\alpha_{13}\\
\alpha_{23}
\end{pmatrix}+\begin{pmatrix}G_{0,0}\epsilon^{-1}\\
(G_{2,-2}+G_{2,2})\epsilon^{-1}\\
(G_{2,-1}-G_{2,1})\epsilon^{-1}\\
G_{2,0}\epsilon^{-1}\\
i(G_{2,-1}+G_{2,1})\epsilon^{-1}\\
i(G_{2,-2}-G_{2,2})\epsilon^{-1}
\end{pmatrix}-\begin{pmatrix}g_{0,0}^{\dce}\\
g_{2,-2}^{\dce}+g_{2,2}^{\dce}\\
g_{2,-1}^{\dce}-g_{2,1}^{\dce}\\
g_{2,0}^{\dce}\\
i(g_{2,-1}^{\dce}+g_{2,1}^{\dce})\\
i(g_{2,-2}^{\dce}-g_{2,2}^{\dce})
\end{pmatrix}=0,\label{eq:alpha ij system}
\end{equation}
where $G[\{\alpha_{ij}\}]\epsilon^{-1}=\sum_{l,m}G_{l,m}\epsilon^{-1}Y_{l,m}=O(\epsilon)$
and 
\[
A=\frac{\sqrt{\pi}}{30}\begin{pmatrix}-10 & -10 & -10 & 0 & 0 & 0\\
-2\sqrt{30} & 2\sqrt{30} & 0 & 0 & 0 & 0\\
0 & 0 & 0 & 0 & -4\sqrt{30} & 0\\
2\sqrt{5} & 2\sqrt{5} & -4\sqrt{5} & 0 & 0 & 0\\
0 & 0 & 0 & 0 & 0 & 4\sqrt{30}\\
0 & 0 & 0 & 4\sqrt{30} & 0 & 0
\end{pmatrix}
\]
is an invertible matrix. It is readily seen that the Jacobian matrix
of $\mathbf{f}$ with respect to $\{\alpha_{ij}\}$ is $A$ at $\epsilon=0$.
Therefore, by the implicit function theorem, the system (\ref{eq:alpha ij system})
has a a unique solution $\{\alpha_{ij}\}$ for a sufficiently small
$\epsilon>0$. Moreover,
\[
\begin{pmatrix}\alpha_{11}\\
\alpha_{22}\\
\alpha_{33}\\
\alpha_{12}\\
\alpha_{13}\\
\alpha_{23}
\end{pmatrix}=A^{-1}\begin{pmatrix}g_{0,0}^{\dce}\\
g_{2,-2}^{\dce}+g_{2,2}^{\dce}\\
g_{2,-1}^{\dce}-g_{2,1}^{\dce}\\
g_{2,0}^{\dce}\\
i(g_{2,-1}^{\dce}+g_{2,1}^{\dce})\\
i(g_{2,-2}^{\dce}-g_{2,2}^{\dce})
\end{pmatrix}+O(\epsilon).
\]

\end{proof}

\subsection{The proof of Theorem \ref{thm:existence of ellipsoids}}

The system (\ref{eq:equiv ellipsoid vapor})-(\ref{eq:equiv ellipsoid solution form})
is closely related to the so-called conductor problem for the ellipsoid,
whose solution can be found in \cite{Kellogg:1967uz}:
\begin{prop}
\label{prop:kellogg}\cite{Kellogg:1967uz} Let $\partial\Omega$
be an ellipsoid centered at the origin, which has the semi-axes are
of lengths $\mathfrak{a}$, $\mathfrak{b}$, and $\mathfrak{c}$:
\[
\partial\Omega=\left\{ (x,y,z)\in\mathbb{R}^{3}\colon\frac{x^{2}}{\mathfrak{a}^{2}}+\frac{y^{2}}{\mathfrak{b}^{2}}+\frac{z^{2}}{\mathfrak{c}^{2}}=1\right\} .
\]
Then there exists a unique solution $\phi$ to the following overdetermined
problem: 
\begin{align}
 & \Delta\phi=0 &  & \text{in the exterior of}\,\,\partial\Omega,\nonumber \\
 & \phi=1 &  & \text{on}\,\,\partial\Omega,\nonumber \\
 & \phi\rightarrow0 &  & \text{as}\,\, r\rightarrow\infty,\nonumber \\
 & \frac{\partial\phi}{\partial\mathbf{n}}=-\frac{E}{\mathfrak{abc}}d_{0}(x,y,z) &  & \text{on}\,\,\partial\Omega,\label{eq:kellogg kinematic}
\end{align}
where $d_{0}(x,y,z)$ is the distance from the center of the ellipsoid
to the plane tangent to $\partial\Omega$ at $(x,y,z)$ such that
\[
d_{0}(x,y,z)=\left(\frac{x^{2}}{\mathfrak{a}^{4}}+\frac{y^{2}}{\mathfrak{b}^{4}}+\frac{z^{2}}{\mathfrak{c}^{4}}\right)^{-\frac{1}{2}},
\]
and $E$ is the capacity of the ellipsoid, which satisfies 
\begin{equation}
1=\frac{E}{2}\int_{0}^{\infty}\frac{ds}{\sqrt{(\mathfrak{a}^{2}+s)(\mathfrak{b}^{2}+s)(\mathfrak{c}^{2}+s)}}.
\end{equation}

\end{prop}
We now give the proof of Theorem \ref{thm:existence of ellipsoids}.
\begin{proof}[Proof of Theorem \ref{thm:existence of ellipsoids}]
By Lemma \ref{lem:if then ellipsoid}, the solution to (\ref{eq:dce stokes system})-(\ref{eq:dce boundary condition})
should be such that $\{r=1+\epsilon g^{\mathfrak{E}}\}$ is an ellipsoid.
By Lemma \ref{lem:bijective lem}, we can construct a unique ellipsoid
$\{r=1+\epsilon g^{\mathfrak{E}}\}$ for given $g_{0,0}^{\dce}$ and
$\left\{ g_{2,m}^{\dce}\right\} _{|m|\leq2}$. Therefore the remaining
part is to show the well-posedness of (\ref{eq:dce stokes system})-(\ref{eq:dce boundary condition})
with the given ellipsoid $\{r=1+\epsilon g^{\mathfrak{E}}\}$. Using
the proof of Lemma \ref{lem:if then ellipsoid}, we conclude that
(\ref{eq:dce vapor})-(\ref{eq:dce boundary condition}) is equivalent
to the following system: 
\begin{equation}
\begin{cases}
\Delta\phi=0, & r>R(t)(1+\epsilon g^{\dce}),\\
\phi=1, & r=R(t)(1+\epsilon g^{\dce}),\\
\phi\rightarrow0 & \textnormal{as}\, r\rightarrow\infty,
\end{cases}\label{eq:proof exist of ellip vapor sys1}
\end{equation}
\begin{equation}
v_{\mathbf{n}}=\frac{1}{2}\frac{\partial\phi}{\partial\mathbf{n}},\quad r=R(t)(1+\epsilon g^{\dce}),\label{eq:proof ellip kinematic}
\end{equation}
\[
\phi=R(t)/r+\epsilon\Phi^{\dce},
\]
where $R(t)=[1-t/(1+\epsilon t_{0}Y_{0,0})]^{1/2}$. Moreover, by
the scaling-property, (\ref{eq:proof exist of ellip vapor sys1})
reduces to 
\begin{equation}
\begin{cases}
\Delta\phi=0, & r>1+\epsilon g^{\dce},\\
\phi=1, & r=1+\epsilon g^{\dce},\\
\phi\rightarrow0 & \textnormal{as}\, r\rightarrow\infty.
\end{cases}\label{eq:rescaling vapor sys1}
\end{equation}
We now deduce a detailed calculation for (\ref{eq:proof ellip kinematic}).
By the symmetry of (\ref{eq:rescaling vapor sys1}) under rotation
and translation, it is sufficient to consider $\{r=1+\epsilon g^{\dce}\}$
as a basic ellipsoid $\left\{ \frac{x^{2}}{\mathfrak{a}^{2}}+\frac{y^{2}}{\mathfrak{b}^{2}}+\frac{z^{2}}{\mathfrak{c}^{2}}=1\right\} $.
Let $\mathbf{v}=(x,y,z)\in\mathbb{R}^{3}$ be a vector field satisfying
$\frac{x^{2}}{\mathfrak{a}^{2}}+\frac{y^{2}}{\mathfrak{b}^{2}}+\frac{z^{2}}{\mathfrak{c}^{2}}-1=0$.
Then (\ref{eq:proof ellip kinematic}) becomes 
\begin{equation}
v_{\mathbf{n}}=\left(\frac{d}{dt}R(t)\mathbf{v}\right)\cdot\mathbf{n}=\frac{1}{2}\frac{\partial\phi}{\partial\mathbf{n}},\quad r=R(t)(1+\epsilon g^{\dce}).\label{eq:ellip vector kinematc}
\end{equation}
By the scaling-property of (\ref{eq:ellip vector kinematc}), we are
led to the following relation holding on $\{r=1+\epsilon g^{\dce}\}$:
\begin{equation}
\frac{\partial\phi}{\partial\mathbf{n}}=\frac{dR^{2}}{dt}\mathbf{v}\cdot\mathbf{n}=-\frac{1}{1+\epsilon t_{0}Y_{0,0}}\mathbf{v}\cdot\mathbf{n}=-\frac{1}{1+\epsilon t_{0}Y_{0,0}}\mathbf{v}\cdot\frac{\nabla\left(\frac{x^{2}}{\mathfrak{a}^{2}}+\frac{y^{2}}{\mathfrak{b}^{2}}+\frac{z^{2}}{\mathfrak{c}^{2}}-1\right)}{\left\vert \nabla\left(\frac{x^{2}}{\mathfrak{a}^{2}}+\frac{y^{2}}{\mathfrak{b}^{2}}+\frac{z^{2}}{\mathfrak{c}^{2}}-1\right)\right\vert }=-\frac{d_{0}(x,y,z)}{1+\epsilon t_{0}Y_{0,0}}.\label{eq:rescaling kinematic}
\end{equation}
If we set 
\begin{equation}
t_{0}=\frac{1}{\epsilon Y_{0,0}}\left(\frac{\mathfrak{abc}}{E}-1\right),\label{eq:set t0 be}
\end{equation}
(\ref{eq:rescaling kinematic}) is equivalent to (\ref{eq:kellogg kinematic}),
and hence by Proposition \ref{prop:kellogg}, the problem (\ref{eq:dce vapor})-(\ref{eq:dce boundary condition})
has a unique solution $(\Phi^{\mathfrak{E}},t_{0})$. Here, we can
show the well-posedness of (\ref{eq:dce stokes system}) in a similar
way as in \cite{Friedman:2002io} and \cite{Friedman:2002en}. Therefore
the overdetermined problem (\ref{eq:dce stokes system})-(\ref{eq:dce boundary condition})
has a unique solution $(P^{\mathfrak{E}},\mathbf{U}^{\mathfrak{E}},\Phi^{\mathfrak{E}},g^{\mathfrak{E}},t_{0})$.
Moreover, by (\ref{eq:P1 lm})-(\ref{eq:V1 1m}), (\ref{eq:X1 -11})-(\ref{eq:V1 00}),
(\ref{eq:vaporp sol lm}), (\ref{eq:ellipsoid r function Ylm}), and
(\ref{eq:set t0 be}), one can see that the leading-order terms of
$(P^{\mathfrak{E}},\mathbf{U}^{\mathfrak{E}},\Phi^{\mathfrak{E}},g^{\mathfrak{E}},t_{0})$
consist of $g_{0,0}^{\dce}$ and $\left\{ g_{2,m}^{\dce}\right\} _{|m|\leq2}$.
Using this fact crucially, following a similar approach as in \cite{Friedman:2002io},
\cite{Friedman:2002en}, and using standard elliptic estimates, we
immediately deduce that (\ref{eq:estimate1}) holds.
\end{proof}

\section{\label{sec:The linear evolution problem}The linear evolution problem}

In this section we prove Theorem \ref{thm:the linear evolution problem}.
By Lemmas \ref{lem:Uer r 1} and \ref{lem:Phir r 1}, we conclude
that (\ref{eq:dct re kinematic}) and (\ref{eq:dct re initial}) for
the auxiliary linear problem (\ref{eq:aux linear problem}) reduce
to the following system of ordinary differential equations on $r=1$
for $g_{l,m}^{\mathfrak{T}}$: 
\begin{equation}
\left(g_{l,m}^{\dct}\right)_{\tau}+\left(l-2\right)g_{l,m}^{\dct}+2e^{-\tau}\sigma B(l)g_{l,m}^{\dct}=2e^{-\tau}J_{l,m}^{\dct}+K_{l,m}^{\dct}+F_{l,m}^{9,10;\mathfrak{T}},\quad l\neq2,\label{eq:g dct ODE l neq 2}
\end{equation}
\begin{equation}
\left(g_{2,m}^{\mathfrak{T}}\right)_{\tau}+2e^{-\tau}\sigma B(2)g_{2,m}^{\dct}+e^{-\tau}\sigma\frac{20}{19}g_{2,m}^{\mathfrak{E}}=2e^{-\tau}J_{2,m}^{\dct}+K_{2,m}^{\dct}+F_{2,m}^{9,10;\mathfrak{T}},\label{eq:g dct ODE l 2}
\end{equation}
\[
B(l)=\begin{cases}
(2l^{3}+5l^{2}+2l)/(8l^{2}+16l+12) & \textnormal{if}\quad l\geq2,\\
0 & \textnormal{if}\quad l=0,1,
\end{cases}
\]
with the initial condition 
\begin{equation}
g_{l,m}^{\mathfrak{T}}(\tau=0)-g_{l,m}^{0}=F_{l,m}^{11;\dct},\quad l>2,\label{eq:gdct initial l}
\end{equation}
\begin{equation}
g_{2,m}^{\dct}(\tau=0)-g_{2,m}^{0}+g_{2,m}^{\mathfrak{E}}=F_{2,m}^{11;\dct},\quad|m|\leq2,\label{eq:gdct initial l 2}
\end{equation}
\begin{equation}
g_{1,m}^{\dct}(\tau=0)-g_{1,m}^{0}+b_{1,m}=F_{1,m}^{11;\dct},\quad|m|\leq1,\label{eq:gdct initial l 1}
\end{equation}
\begin{equation}
g_{0,0}^{\dct}(\tau=0)-g_{0,0}^{0}+g_{0,0}^{\dce}=F_{0,0}^{11;\dct},\label{eq:gdct initial l 0}
\end{equation}
where $J_{l,m}^{\dct}$ are given in terms of $F_{l,m}^{1;\dct}$,
$\mathbf{F}_{l,m}^{2;\dct;V}$, $\mathbf{F}_{l,m}^{2;\dct;X}$, $\mathbf{F}_{l,m}^{2;\dct;W}$,
$F_{l,m}^{3;\dct}$, $\mathbf{F}_{l,m}^{4;\dct;V}$, $\mathbf{F}_{l,m}^{4;\dct;X}$,
$\mathbf{F}_{l,m}^{4;\dct;W}$, $\mathbf{F}_{j}^{5;\dct}$, and $\mathbf{F}_{j}^{6;\dct}$
and $K_{l,m}^{\dct}$ consists of $F_{l,m}^{7;\dct}$ and $F_{l,m}^{8;\dct}$.
We will determine the unknowns $g_{0,0}^{\dce}$, $\left\{ b_{1,n}\right\} _{|n|\leq1}$,
and $\left\{ g_{2,m}^{\dce}\right\} _{|m|\leq2}$ uniquely such that
$g_{0,0}^{\dct}$, $\left\{ g_{1,n}^{\dct}\right\} _{|n|\leq1}$,
and $\left\{ g_{2,m}^{\dct}\right\} _{|m|\leq2}$ go to zero as $\tau\rightarrow\infty$,
which is essential to verify $\|\mathsf{x}\|_{\mathscr{X}}<\infty$.

In order to establish estimates for (\ref{eq:g dct ODE l neq 2})-(\ref{eq:gdct initial l 0})
we will use the following lemma and its proof, which are slightly
different from Lemma 3.1 in \cite{Friedman:2001vk}:
\begin{lem}
\label{lem:ode lemma}Consider the initial value problem 
\begin{equation}
\frac{dy(\tau)}{d\tau}+\lambda y(\tau)+e^{-\tau}\gamma y(\tau)=f(\tau),\quad\tau>0,\label{eq:s61}
\end{equation}
\[
y(0)=y_{0},
\]
where $\gamma>0$ and $f\in L^{2}(0,T)$ for any $T>0$. If $0<\lambda_{0}<\lambda$,
then the following inequalities hold for all $\tau>0$: 
\begin{equation}
\int_{0}^{\tau}e^{2\lambda_{0}s}y^{2}(s)ds\leq\frac{2}{\left(\lambda-\lambda_{0}\right)^{2}}\int_{0}^{\tau}e^{2\lambda_{0}s}f^{2}(s)ds+\frac{y_{0}^{2}}{\lambda-\lambda_{0}},\label{eq:s62}
\end{equation}
\begin{equation}
\int_{0}^{\tau}e^{2\lambda_{0}s}\left(\frac{dy(s)}{ds}\right)^{2}ds\leq C\left(\frac{2\lambda^{2}}{\left(\lambda-\lambda_{0}\right)^{2}}+1\right)\int_{0}^{\tau}e^{2\lambda_{0}s}f^{2}(s)ds+\frac{C\lambda^{2}y_{0}^{2}}{\lambda-\lambda_{0}},\label{eq:s63}
\end{equation}
where $C$ is some constant.\end{lem}
\begin{proof}
We conclude 
\[
y(\tau)=\int_{0}^{\tau}e^{\gamma(e^{-\tau}-e^{-s})}e^{-\lambda(\tau-s)}f(s)ds+e^{\gamma(e^{-\tau}-1)}y_{0}e^{-\lambda\tau}.
\]
Hence 
\begin{align*}
 & y^{2}(\tau)e^{2\lambda_{0}\tau}\\
 & \leq2\left(\int_{0}^{\tau}e^{-(\lambda-\lambda_{0})(\tau-s)}ds\right)\int_{0}^{\tau}e^{-(\lambda-\lambda_{0})(\tau-s)}(e^{\alpha s}f(s))^{2}ds+2y_{0}^{2}e^{-2(\lambda-\lambda_{0})\tau}\\
 & \leq\frac{2}{\lambda-\lambda_{0}}\int_{0}^{\tau}e^{-(\lambda-\lambda_{0})(\tau-s)}(e^{\alpha s}f(s))^{2}ds+2y_{0}^{2}e^{-2(\lambda-\lambda_{0})\tau}
\end{align*}
and, by integration over time, 
\begin{align*}
\int_{0}^{\tau}y^{2}(u)e^{2\lambda_{0}u}du & \leq\frac{2}{\lambda-\lambda_{0}}\int_{0}^{\tau}\int_{0}^{u}e^{-(\lambda-\lambda_{0})(u-s)}(e^{\lambda_{0}s}f(s))^{2}dsdu+\frac{y_{0}^{2}}{\lambda-\lambda_{0}}\\
 & \leq\frac{2}{(\lambda-\lambda_{0})^{2}}\int_{0}^{\tau}e^{2\lambda_{0}s}f^{2}(s)ds+\frac{y_{0}^{2}}{\lambda-\lambda_{0}},
\end{align*}
which implies (\ref{eq:s62}). The inequality (\ref{eq:s63}) follows
easily from (\ref{eq:s62}) by using (\ref{eq:s61}).
\end{proof}
We now give the proof of Theorem \ref{thm:the linear evolution problem}.
\begin{proof}[Proof of Theorem \ref{thm:the linear evolution problem}]
We first solve the system (\ref{eq:dct system}) for the auxiliary
linear problem (\ref{eq:aux linear problem}) in the same way as Section
\ref{sec:the ODE system}. If the interface functions $g^{\dce}$
and $g^{\dct}$ are known, then (\ref{eq:dct system}) forms an elliptic
system and one can show the well-posedness by elliptic theory. Therefore
it is sufficient to consider (\ref{eq:g dct ODE l neq 2})-(\ref{eq:gdct initial l 0}).

We hereafter introduce a key ingredient in the proof. In the case
$l=0$, we are led to 
\begin{align}
g_{0,0}^{\mathfrak{T}} & =e^{2\tau}\left(g_{0,0}^{0}-g_{0,0}^{\mathfrak{E}}+F_{0,0}^{11;\mathfrak{T}}+\int_{0}^{\infty}e^{-2s}(2e^{-s}J_{0,0}^{\mathfrak{T}}+K_{0,0}^{\mathfrak{T}}+F_{0,0}^{9,10;\mathfrak{T}})ds\right)\label{eq:g dct l 0 sol}\\
 & -e^{2\tau}\int_{\tau}^{\infty}e^{-2s}(2e^{-s}J_{0,0}^{\mathfrak{T}}+K_{0,0}^{\mathfrak{T}}+F_{0,0}^{9,10;\mathfrak{T}})ds.\nonumber 
\end{align}
Setting 
\begin{equation}
g_{0,0}^{\mathfrak{E}}=g_{0,0}^{0}+F_{0,0}^{11;\mathfrak{T}}+\int_{0}^{\infty}e^{-2s}(2e^{-s}J_{0,0}^{\mathfrak{T}}+K_{0,0}^{\mathfrak{T}}+F_{0,0}^{9,10;\mathfrak{T}})ds,\label{eq:g dce 00}
\end{equation}
we deduce from (\ref{eq:g dct l 0 sol}) that 
\begin{equation}
g_{0,0}^{\mathfrak{T}}=-e^{2\tau}\int_{\tau}^{\infty}e^{-2s}(2e^{-s}J_{0,0}^{\mathfrak{T}}+K_{0,0}^{\mathfrak{T}}+F_{0,0}^{9,10;\mathfrak{T}})ds.\label{eq:g dct l=00003D00003D00003D0 sol re}
\end{equation}
We will use the following inequality: let $y(\tau)=-e^{2\tau}\int_{\tau}^{\infty}e^{-2s}f(s)ds$,
then by using Minkowski's inequality for integrals, one can check
that 
\begin{align}
\left(\int_{0}^{\tau}e^{2\lambda_{0}s}y^{2}ds\right)^{1/2} & =\left(\int_{0}^{\tau}\left(\int_{0}^{\infty}e^{-(2+\lambda_{0})u}e^{\lambda_{0}(s+u)}f(s+u)du\right)^{2}ds\right)^{1/2}\nonumber \\
 & \leq\int_{0}^{\infty}e^{-(2+\lambda_{0})u}\left(\int_{u}^{\tau+u}\left(e^{\lambda_{0}s}f(s)\right)^{2}ds\right)^{1/2}du\nonumber \\
 & \leq C\left(\int_{0}^{\infty}e^{2\lambda_{0}s}f^{2}(s)ds\right)^{1/2}.\label{eq:l 0 ineq}
\end{align}
By applying the inequality (\ref{eq:l 0 ineq}) to (\ref{eq:g dct l=00003D00003D00003D0 sol re}),
we conclude 
\begin{equation}
\int_{0}^{\tau}e^{2\lambda_{0}s}|g_{0,0}^{\mathfrak{T}}|^{2}ds\leq C\left(\int_{0}^{\infty}e^{2\lambda_{0}s}(|e^{-s}J_{0,0}^{\dct}|^{2}+|K_{0,0}^{\dct}|^{2}+|F_{0,0}^{9,10;\dct}|^{2})ds\right).\label{eq:est l 0}
\end{equation}
In the case $l=1$, following the same argument as in the case $l=0$,
we have 
\begin{equation}
b_{1,m}=g_{1,m}^{0}+F_{1,m}^{11;\mathfrak{T}}+\int_{0}^{\infty}e^{-s}(2e^{-s}J_{1,m}^{\mathfrak{T}}+K_{1,m}^{\mathfrak{T}}+F_{1,m}^{9,10;\mathfrak{T}})ds,\label{eq:b1m}
\end{equation}
and hence 
\begin{equation}
\int_{0}^{\tau}e^{2\lambda_{0}s}|g_{1,m}^{\mathfrak{T}}|^{2}ds\leq C\left(\int_{0}^{\infty}e^{2\lambda_{0}s}(|e^{-s}J_{1,m}^{\dct}|^{2}+|K_{1,m}^{\dct}|^{2}+|F_{1,m}^{9,10;\dct}|^{2})ds\right).\label{eq:est l 1}
\end{equation}
In the case $l=2$, we are led to 
\begin{equation}
g_{2,m}^{\mathfrak{T}}=e^{\sigma\frac{20}{19}(e^{-\tau}-1)}(g_{2,m}^{0}+F_{2,m}^{11;\mathfrak{T}})-g_{2,m}^{\mathfrak{E}}+\int_{0}^{\tau}e^{\sigma\frac{20}{19}(e^{-\tau}-e^{-s})}(e^{-s}J_{2,m}^{\mathfrak{T}}+K_{2,m}^{\mathfrak{T}}+F_{2,m}^{9,10;\mathfrak{T}})ds.\label{eq:g dct solution l 2}
\end{equation}
Setting 
\begin{equation}
g_{2,m}^{\dce}=e^{-\sigma\frac{20}{19}}(g_{2,m}^{0}+F_{2,m}^{11;\mathfrak{T}})+\int_{0}^{\infty}e^{-\sigma\frac{20}{19}e^{-s}}(e^{-s}J_{2,m}^{\mathfrak{T}}+K_{2,m}^{\mathfrak{T}}+F_{2,m}^{9,10;\mathfrak{T}})ds,\label{eq:g dce 2m}
\end{equation}
we conclude 
\begin{equation}
g_{2,m}^{\dct}\rightarrow0\quad\textnormal{as}\quad\tau\rightarrow\infty.\label{eq:g dct decaying condition}
\end{equation}
As a result, differentiating (\ref{eq:g dct solution l 2}) in $\tau$,
we have 
\begin{align*}
\left(g_{2,m}^{\mathfrak{T}}\right)_{\tau} & =-\sigma\frac{20}{19}e^{-\tau}e^{\sigma\frac{20}{19}(e^{-\tau}-1)}(g_{2,m}^{0}+F_{2,m}^{11;\mathfrak{T}})\\
 & -\int_{0}^{\tau}\sigma\frac{20}{19}e^{-\tau}e^{\sigma\frac{20}{19}(e^{-\tau}-e^{-s})}(e^{-s}J_{2,m}^{\mathfrak{T}}+K_{2,m}^{\mathfrak{T}}+F_{2,m}^{9,10;\mathfrak{T}})ds\\
 & +e^{-\tau}J_{2,m}^{\mathfrak{T}}+K_{2,m}^{\mathfrak{T}}+F_{2,m}^{9,10;\mathfrak{T}}
\end{align*}
so that, by the same arguments as in Lemma \ref{lem:ode lemma}, 
\begin{equation}
\int_{0}^{\tau}e^{2\lambda_{0}s}\left\vert \left(g_{2,m}^{\mathfrak{T}}\right)_{s}\right\vert ^{2}ds\leq C\left(|g_{2,m}^{0}|^{2}+|F_{2,m}^{11;\mathfrak{T}}|^{2}+\int_{0}^{\tau}e^{2\lambda_{0}s}(|e^{-s}J_{2,m}^{\dct}|^{2}+|K_{2,m}^{\dct}|^{2}+|F_{2,m}^{9,10;\dct}|^{2})ds\right).\label{eq:ineq65}
\end{equation}
By using (\ref{eq:g dct decaying condition}) and the inequality (4.46)
in \cite{Friedman:2001vk}, we deduce from (\ref{eq:ineq65}) that
\begin{align}
\int_{0}^{\tau}e^{2\lambda_{0}s}|g_{2,m}^{\mathfrak{T}}|^{2}ds & \leq\int_{0}^{\infty}e^{-2\lambda_{0}u}\int_{u}^{\tau+u}e^{2\lambda_{0}s}\left\vert \left(g_{2,m}^{\mathfrak{T}}\right)_{s}\right\vert ^{2}dsdu\notag\\
 & \leq C\left(|g_{2,m}^{0}|^{2}+|F_{2,m}^{11;\mathfrak{T}}|^{2}+\int_{0}^{\infty}e^{2\lambda_{0}s}(|e^{-s}J_{2,m}^{\dct}|^{2}+|K_{2,m}^{\dct}|^{2}+|F_{2,m}^{9,10;\dct}|^{2})ds\right).\label{eq:est l 2}
\end{align}
In the case $l>2$, it immediately follows from Lemma \ref{lem:ode lemma}
that 
\begin{equation}
\int_{0}^{\tau}e^{2\lambda_{0}s}|g_{l,m}^{\mathfrak{T}}|^{2}ds\leq\frac{C}{(l-2-\lambda_{0})^{2}}\int_{0}^{\tau}e^{2\lambda_{0}s}(|e^{-s}J_{l,m}^{\dct}|^{2}+|K_{l,m}^{\dct}|^{2}+|F_{l,m}^{9,10;\dct}|^{2})ds+C\frac{|g_{l,m}^{0}|^{2}+|F_{l,m}^{11;\mathfrak{T}}|^{2}}{l-2-\lambda_{0}}.\label{eq:est l bigger 2}
\end{equation}
The derivatives of $g_{l,m}^{\dct}$ in $\tau$ can be estimated similarly
for all $l\geq0$. Therefore, by using the estimates (\ref{eq:est l 0}),
(\ref{eq:est l 1}), (\ref{eq:est l 2}), (\ref{eq:est l bigger 2})
and the equalities (\ref{eq:g dce 00}), (\ref{eq:b1m}), (\ref{eq:g dce 2m}),
we can deduce the following estimate in a similar way as in \cite{Friedman:2001vk}:
\begin{equation}
\left\Vert g^{\mathfrak{T}}\right\Vert _{\frac{13}{2},\mathbb{S}}+\left\vert g_{0,0}^{\mathfrak{E}}\right\vert +\sum_{|n|\leq1}\left\vert b_{1,n}\right\vert +\sum_{|m|\leq2}\left\vert g_{2,m}^{\mathfrak{E}}\right\vert \leq C\left(\left\Vert g^{0}\right\Vert _{H^{6}(\mathbb{S})}+\left\Vert \mathsf{y}\right\Vert _{\mathscr{Y}}\right).\label{eq:g linear estimate}
\end{equation}
Moreover, from standard elliptic estimates for $(P^{\dct},\mathbf{U}^{\dct},\Phi^{\dct})$,
we immediately conclude that (\ref{eq:est linear problem}) holds.\end{proof}
\begin{rem}
By the Sobolev embedding theorem, the estimate (\ref{eq:g linear estimate})
implies that $g^{\dct}$ is uniformly bounded in space and time as
well as $(P^{\dct},\mathbf{U}^{\dct},\Phi^{\dct})$.
\end{rem}

\begin{rem}
\label{lambda0 restriction}The number $\lambda_{0}$ in the definition
of the Banach spaces $\mathscr{X}$ and $\mathscr{Y}$ should be restricted
to $0<\lambda_{0}<1$ to satisfy the estimates (\ref{eq:est l 0}),
(\ref{eq:est l 1}), (\ref{eq:est l 2}), and (\ref{eq:est l bigger 2}).
\end{rem}
We now give the proof of Theorem \ref{thm:stable ellipsoidal collapse to a point}.
\begin{proof}[Proof of Theorem \ref{thm:stable ellipsoidal collapse to a point}]
By Theorem \ref{thm:the nonlinear evolution problem}, the nonlinear
problem $\mathsf{x}=\mathcal{A}[\mathsf{x}]$ has a unique solution
\[
\mathsf{x}=\left(P^{\mathfrak{T}},\mathbf{U}^{\mathfrak{T}},\Phi^{\mathfrak{T}},g^{\mathfrak{T}},g_{0,0}^{\mathfrak{E}},\left\{ b_{1,n}\right\} _{\left\vert n\right\vert \leq1},\left\{ g_{2,m}^{\mathfrak{E}}\right\} _{\left\vert m\right\vert \leq2}\right)
\]
in the space $\mathscr{X}$, so that the solution $\mathsf{x}$ to
the evolution problem (\ref{eq:dct system})-(\ref{eq:dct initial condition})
exists uniquely. Indeed, by Theorem \ref{thm:existence of ellipsoids},
the parameters $g_{0,0}^{\mathfrak{E}}$ and $\{g_{2,m}^{\mathfrak{E}}\}_{|m|\leq2}$
of the solution $\mathsf{x}$ construct a unique ellipsoid $\{r=1+\epsilon g^{\dce}\}$
as well as a unique solution $(P^{\mathfrak{E}},\mathbf{U}^{\mathfrak{E}},\Phi^{\mathfrak{E}},g^{\mathfrak{E}},t_{0})$
to the ellipsoidal problem (\ref{eq:dce stokes system})-(\ref{eq:dce boundary condition}).
We give a simple diagram to explain the procedure that transforms
the free-boundary problem (\ref{eq:ED system primary})-(\ref{eq:ED initial condition primary})
into $\mathsf{x}=\mathcal{A}[\mathsf{x}]$ of Theorem \ref{thm:the nonlinear evolution problem}:

\begin{center}
\begin{tikzpicture}[->,>=stealth']

\node[state](FBP)
{
\begin{tabular}{l}  
The free-boundary problem (\ref{eq:ED system primary})-(\ref{eq:ED initial condition primary})
\end{tabular}
};

\node[state, node distance=5cm, right of=FBP, yshift=-1cm](FIXEDDOMAIN)
{
\begin{tabular}{l}
The fixed-domain problem (\ref{eq:ED system fixed domain}), (\ref{eq:ED kinematic condition fixed domain}), (\ref{eq:ED initial condition fixed domain})
\end{tabular}
};

\node[state,  below of=FIXEDDOMAIN,  node distance=2cm](DECOMPOSITION)
{
\begin{tabular}{l}
The ellipsoidal part $\dce$ (\ref{eq:dce stokes system})-(\ref{eq:dce boundary condition})\\
The evolution part $\dct$ (\ref{eq:dct system})-(\ref{eq:dct initial condition})
\end{tabular}
};

\node[state,  below of=DECOMPOSITION,  node distance=3cm](NONLINEAR)
{
\begin{tabular}{l}
$\mathsf{x}=\mathcal{A}[\mathsf{x}]$ of Theorem \ref{thm:the nonlinear evolution problem}
\end{tabular}
};

\path 

(FBP) edge[bend left=50]  node[above]
{
\begin{tabular}{l}
Apply the self-similar variables (\ref{eq:self-similar variables solution}), (\ref{eq:self-similar variables})\\
and the Hanzawa transformation (\ref{eq:hanzawa transformation})
\end{tabular}
} (FIXEDDOMAIN)  

(FIXEDDOMAIN) edge               node[left]
{
Decompose
} (DECOMPOSITION) 

(DECOMPOSITION)  edge                    node[left]                  
{
\begin{tabular}{l}
By Theorem \ref{thm:existence of ellipsoids},\\
the parameters $g_{0,0}^{\dce}$ and $\{g_{2,m}^{\dce}\}_{|m|\leq2}$ of the solution $\mathsf{x}$\\
construct the ellipsoidal part uniquely
\end{tabular}
} (NONLINEAR)

; 

\end{tikzpicture}
\end{center}Therefore, the fixed-domain problem (\ref{eq:ED system fixed domain}),
(\ref{eq:ED kinematic condition fixed domain}), (\ref{eq:ED initial condition fixed domain})
has a unique solution of the form
\[
\left(P,\mathbf{U},\Phi,g\right)=(P^{\mathfrak{E}},\mathbf{U}^{\mathfrak{E}},\Phi^{\mathfrak{E}},g^{\mathfrak{E}})+(P^{\mathfrak{T}},\mathbf{U}^{\mathfrak{T}},\Phi^{\mathfrak{T}},g^{\mathfrak{T}})\quad\textnormal{and}\quad(\mathbf{x}_{0},t_{0}).
\]
We remark that $(\mathbf{x}_{0},t_{0})$ is a part of the solution,
where $\left\{ b_{1,n}\right\} _{\left\vert n\right\vert \leq1}$
determine $\mathbf{x}_{0}$ uniquely by (\ref{eq:x0 b1 relation}).
Finally, by the norm of $\mathscr{X}$, we have $\|g-g^{\dce}\|_{\frac{13}{2},\mathbb{S}}<\infty$,
as well as the corresponding $\left(P,\mathbf{U},\Phi\right)$, i.e.
\[
\left[\int_{0}^{\infty}e^{2\lambda_{0}\tau}\sum_{k=0}^{2}\left\Vert \frac{\partial^{k}}{\partial\tau^{k}}\left(g-g^{\dce}\right)\right\Vert _{H^{\frac{13}{2}-k}(\mathbb{S})}^{2}d\tau\right]^{1/2}<\infty,
\]
where $0<\lambda_{0}<1$ (see Remark \ref{lambda0 restriction}).
From here, reversing the transformation (\ref{eq:self-similar variables solution}),
(\ref{eq:self-similar variables}), (\ref{eq:hanzawa transformation})
and using the Sobolev embedding theorem, we easily conclude that the
solution to the problem (\ref{eq:ED system primary})-(\ref{eq:ED initial condition primary})
exists uniquely and (\ref{eq:thm1.1 free boundary solution}) and
(\ref{eq:thm1.1 converging to ellipsoid}) in Theorem \ref{thm:stable ellipsoidal collapse to a point}
are valid.

We now show the estimates (\ref{eq:x0 t0 thm1.1}) and (\ref{eq:gdce thm1.1})
in Theorem \ref{thm:stable ellipsoidal collapse to a point}. By using
Theorem \ref{thm:existence of ellipsoids} and the formulas (\ref{in1})-(\ref{eq:ED initial condition fixed domain}),
we construct an ellipsoid represented by 
\begin{equation}
\{r=1+\epsilon g_{0,0}^{0}Y_{0,0}+\epsilon\sum_{|n|\leq1}g_{1,n}^{0}Y_{1,n}+\epsilon\sum_{|m|\leq2}e^{-\sigma\frac{20}{19}}g_{2,m}^{0}Y_{2,m}+O(\epsilon^{2})\},\label{eq:g0 ellipsoid}
\end{equation}
where the $O(\epsilon^{2})$-terms merely depend on $\{g_{l,m}^{0}\}_{l\leq2}$.
Let $\epsilon\tilde{\mathbf{x}}_{0}$ be the center of mass of volume
enclosed by (\ref{eq:g0 ellipsoid}). Then we are led to the following
relation:
\begin{equation}
\left(\begin{array}{c}
\tilde{x}_{01}\\
\tilde{x}_{02}\\
\tilde{x}_{03}
\end{array}\right)=\sqrt{\frac{3}{8\pi}}\left(\begin{array}{ccc}
1 & 0 & -1\\
-i & 0 & -i\\
0 & \sqrt{2} & 0
\end{array}\right)\left(\begin{array}{c}
g_{1,-1}^{0}\\
g_{1,0}^{0}\\
g_{1,1}^{0}
\end{array}\right),\label{eq:relation center of mass tilde}
\end{equation}
where $\epsilon\tilde{\mathbf{x}}_{0}=\epsilon(\tilde{x}_{01},\tilde{x}_{02},\tilde{x}_{03})$.
The same relation (\ref{eq:x0 b1 relation}) is also satisfied by
$\mathbf{x}_{0}$ and $\{b_{1,n}\}_{|n|\leq1}$. Let $\tilde{t}_{0}$
be the real constant deduced from the capacity formula of (\ref{eq:g0 ellipsoid}):
\begin{equation}
1+\epsilon\tilde{t}_{0}Y_{0,0}=\frac{\tilde{\mathfrak{a}}\tilde{\mathfrak{b}}\tilde{\mathfrak{c}}}{2}\int_{0}^{\infty}\frac{ds}{\sqrt{(\tilde{\mathfrak{a}}^{2}+s)(\tilde{\mathfrak{b}}^{2}+s)(\tilde{\mathfrak{c}}^{2}+s)}},\label{eq:capacity formula tilde}
\end{equation}
where the semi-axes of (\ref{eq:g0 ellipsoid}) are of lengths $\tilde{\mathfrak{a}}$,
$\tilde{\mathfrak{b}}$, and $\tilde{\mathfrak{c}}$. Likewise, the
quantity $t_{0}$ has the same capacity formula (\ref{eq:ellipsoid t0}).
Indeed, from the equalities (\ref{eq:g dce 00}), (\ref{eq:b1m}),
and (\ref{eq:g dce 2m}), we have
\begin{equation}
\begin{cases}
g_{0,0}^{\mathfrak{E}}-g_{0,0}^{0}=O(\epsilon),\\
b_{1,n}-g_{1,n}^{0}=O(\epsilon), & |n|\leq1,\\
g_{2,m}^{\dce}-e^{-\sigma\frac{20}{19}}g_{2,m}^{0}=O(\epsilon), & |m|\leq2.
\end{cases}\label{eq:estimates gdce ginitial}
\end{equation}
Therefore, by using the estimates (\ref{eq:estimates gdce ginitial})
and the relations (\ref{eq:relation center of mass tilde}), (\ref{eq:x0 b1 relation}),
(\ref{eq:capacity formula tilde}), (\ref{eq:ellipsoid t0}), it readily
follows that (\ref{eq:x0 t0 thm1.1}) is valid. Moreover, (\ref{eq:gdce thm1.1})
also follows easily from (\ref{eq:estimates gdce ginitial}).
\end{proof}
\appendix

\section{\label{sec:S Harmonics and Vector S Harmonics}The spherical harmonics
and the vector spherical harmonics}

In this appendix we collect properties of the spherical harmonics
and the vector spherical harmonics. The spherical harmonics $Y_{l,m}(\theta,\varphi)$
are defined on $\mathbb{S}$ by 
\begin{equation}
Y_{l,m}(\theta,\varphi)=(-1)^{m}\sqrt{\frac{(2l+1)(l-m)!}{2(l+m)!}}P_{l,m}(\cos\theta)\frac{e^{im\varphi}}{\sqrt{2\pi}},\quad l\geq0,\quad|m|\leq l,\label{eq:A1}
\end{equation}
where $P_{l,m}(z)=\frac{1}{2^{l}l!}(1-z^{2})^{m/2}\frac{d^{l+m}}{dz^{l+m}}(z^{2}-1)^{l}$.
The spherical harmonics $\{Y_{l,m}\}_{l,m}$ form a complete set in
$L^{2}(\mathbb{S})$ and 
\[
\Delta_{\omega}Y_{l,m}=-l(l+1)Y_{l,m},
\]
where $\Delta_{\omega}$ is the Laplace operator on the surface of
the sphere (see (\ref{eq:laplacian on sphere})). The vector spherical
harmonics are defined by 
\begin{equation}
\mathbf{V}_{l,m}=\left[-\left(\frac{l+1}{2l+1}\right)^{1/2}Y_{l,m}\right]\mathbf{e}_{r}+\left[\frac{1}{(l+1)^{1/2}(2l+1)^{1/2}}\frac{\partial Y_{l,m}}{\partial\theta}\right]\mathbf{e}_{\theta}+\left[\frac{imY_{l,m}}{(l+1)^{1/2}(2l+1)^{1/2}\sin\theta}\right]\mathbf{e}_{\varphi},\label{eq:def Vlm}
\end{equation}
\begin{equation}
\mathbf{X}_{l,m}=\left[\frac{-mY_{l,m}}{l^{1/2}(l+1)^{1/2}\sin\theta}\right]\mathbf{e}_{\theta}+\left[\frac{-i}{l^{1/2}(l+1)^{1/2}}\frac{\partial Y_{l,m}}{\partial\theta}\right]\mathbf{e}_{\varphi},\label{eq:def Xlm}
\end{equation}
\begin{equation}
\mathbf{W}_{l,m}=\left[\left(\frac{l}{2l+1}\right)^{1/2}Y_{l,m}\right]\mathbf{e}_{r}+\left[\frac{1}{l^{1/2}(2l+1)^{1/2}}\frac{\partial Y_{l,m}}{\partial\theta}\right]\mathbf{e}_{\theta}+\left[\frac{imY_{l,m}}{l^{1/2}(2l+1)^{1/2}\sin\theta}\right]\mathbf{e}_{\varphi}.\label{eq:def Wlm}
\end{equation}
The vector spherical harmonics $\{\mathbf{V}_{l,m},\mathbf{X}_{l,m},\mathbf{W}_{l,m}\}_{l,m}$
form an orthonormal basis of $(L^{2}(\mathbb{S}))^{3}$. We need the
following well-known formulas in our proofs (see \cite{Hill:1954up}
and \cite{Friedman:2002en}). We hereafter denote $F(r)$ as a function
of $r$. For the Laplace operator, 
\begin{equation}
\Delta\left[F(r)Y_{l,m}\right]=L_{l}[F]Y_{l,m},\label{eq:laplacian FYlm}
\end{equation}
\begin{equation}
\Delta\left[F(r)\mathbf{V}_{l,m}\right]=L_{l+1}[F]\mathbf{V}_{l,m},\label{eq:laplacian FVlm}
\end{equation}
\begin{equation}
\Delta\left[F(r)\mathbf{X}_{l,m}\right]=L_{l}[F]\mathbf{X}_{l,m},\label{eq:laplacian FXlm}
\end{equation}
\begin{equation}
\Delta\left[F(r)\mathbf{W}_{l,m}\right]=L_{l-1}[F]\mathbf{W}_{l,m},\label{eq:laplacian FWlm}
\end{equation}
where 
\[
L_{l}=\frac{\partial^{2}}{\partial r^{2}}+\frac{2}{r}\frac{\partial}{\partial r}-\frac{l(l+1)}{r^{2}}.
\]
For the gradient operator,

\begin{equation}
\nabla\left[F(r)Y_{l,m}\right]=\left(\frac{l+1}{2l+1}\right)^{1/2}\left(-\frac{\partial}{\partial r}+\frac{l}{r}\right)F\mathbf{V}_{l,m}+\left(\frac{l}{2l+1}\right)^{1/2}\left(\frac{\partial}{\partial r}+\frac{l+1}{r}\right)F\mathbf{W}_{l,m}.\label{eq:gradient FYlm}
\end{equation}
For the divergence operator,

\begin{equation}
\operatorname{div}\left[F(r)\mathbf{V}_{l,m}\right]=-\left(\frac{l+1}{2l+1}\right)^{1/2}\left(\frac{\partial}{\partial r}+\frac{l+2}{r}\right)FY_{l,m},\label{eq:divergence FVlm}
\end{equation}
\begin{equation}
\operatorname{div}\left[F(r)\mathbf{X}_{l,m}\right]=0,\label{eq:divergence FXlm}
\end{equation}
\begin{equation}
\operatorname{div}\left[F(r)\mathbf{W}_{l,m}\right]=\left(\frac{l}{2l+1}\right)^{1/2}\left(\frac{\partial}{\partial r}-\frac{l-1}{r}\right)FY_{l,m}.\label{eq:divergence FWlm}
\end{equation}
Using the following lemmas, we compute the Sobolev norms of spherical
harmonics expansions and vector spherical harmonics expansions:
\begin{lem}[Lemma 8.2 in \cite{Friedman:2001vk}]
Let
\[
F(r,\theta,\varphi)=\sum_{l,m}F_{l,m}(r)Y_{l,m}(\theta,\varphi),
\]
\[
f(\theta,\varphi)=\sum_{l,m}f_{l,m}Y_{l,m}(\theta,\varphi).
\]
Then there exist positive constants $c_{1}$, $c_{2}$ independent
of $F$, $f$ such that
\[
c_{1}\|F\|_{H^{s}(\mathbb{B})}^{2}\leq\sum_{j=0}^{s}\sum_{l\geq0}(1+l^{2(s-j)})\sum_{m}\int_{0}^{1}|D_{r}^{j}F_{l,m}|^{2}r^{2}dr\leq c_{2}\|F\|_{H^{s}(\mathbb{B})}^{2},
\]
\[
c_{1}\|f\|_{H^{s+1/2}(\mathbb{S})}^{2}\leq\sum_{l\geq0}(1+l^{2s+1})\sum_{m}|f_{l,m}|^{2}\leq c_{2}\|f\|_{H^{s+1/2}(\mathbb{S})}^{2},
\]
where $s$ is a nonnegative integer. Moreover, we can define $\|F\|_{H^{s}(\operatorname{ext}\mathbb{B};w)}$
in a similar way.
\end{lem}

\begin{lem}[Lemmas 11.1 and 11.2 in \cite{Friedman:2002en}]
Let
\[
\mathbf{F}(r,\theta,\varphi)=\sum_{l,m}\left(A_{l,m}(r)\mathbf{V}_{l,m}+B_{l,m}(r)\mathbf{X}_{l,m}+C_{l,m}(r)\mathbf{W}_{l,m}\right),
\]
\[
\mathbf{f}(\theta,\varphi)=\sum_{l,m}(\alpha_{l,m}\mathbf{V}_{l,m}+\beta_{l,m}\mathbf{X}_{l,m}+\gamma_{l,m}\mathbf{W}_{l,m}).
\]
Then there exist positive constants $c_{1}$, $c_{2}$ independent
of $\mathbf{F}$, $\mathbf{f}$ such that
\[
c_{1}\|\mathbf{F}\|_{H^{s}(\mathbb{B})}^{2}\leq\sum_{j=0}^{s}\sum_{l\geq0}(1+l^{2(s-j)})\sum_{m}\int_{0}^{1}(|D_{r}^{j}A_{l,m}|^{2}+|D_{r}^{j}B_{l,m}|^{2}+|D_{r}^{j}C_{l,m}|^{2})r^{2}dr\leq c_{2}\|\mathbf{F}\|_{H^{s}(\mathbb{B})}^{2},
\]
\[
c_{1}\|\mathbf{f}\|_{H^{s+1/2}(\mathbb{S})}^{2}\leq\sum_{l\geq0}(1+l^{2s+1})\sum_{m}(|\alpha_{l,m}|^{2}+|\beta_{l,m}|^{2}+|\gamma_{l,m}|^{2})\leq c_{2}\|\mathbf{f}\|_{H^{s+1/2}(\mathbb{S})}^{2},
\]
where $s$ is a nonnegative integer.\end{lem}
\begin{acknowledgement*}
The authors are grateful to the anonymous referee for his/her careful
reading and valuable comments on this paper. Marco A. Fontelos is
partially supported by Grants MTM2011-26016 from the Spanish Ministry
of Science. Hyung Ju Hwang is partially supported by the Basic Science
Research Program (2010-0008127, 2013053914) through the National Research
Foundation of Korea (NRF).
\end{acknowledgement*}
\bibliographystyle{abbrv}
\phantomsection\addcontentsline{toc}{section}{\refname}\bibliography{MAIN}

\end{document}